\documentclass[10pt,leqno]{article}

\usepackage[lite]{amsrefs}
 \usepackage{hyperref}

\usepackage{latexsym,enumerate,pstricks}
\usepackage[notref, notcite]{}

\usepackage{amssymb, amsfonts, eucal,amsmath,amsthm,tabularx}
\usepackage{longtable}

\newtheorem{theorem}{Theorem}[section]
\newtheorem{lemma}[theorem]{Lemma}

 \newtheorem{corollary}[theorem]{Corollary}
\newtheorem{remark}[theorem]{Remark}

\newcommand{\thm}[1]{Theorem~\ref{#1}}
\newcommand{\lem}[1]{Lemma~\ref{#1}}

\newcommand{\z}{{\mathbf{z}}}
\newcommand{\SSigma}{\mathcal{S}}

\newcommand{\Spl}[2]{\SSigma_{#1}(#2)}

\newcommand{\scale}{\vartheta}

 \DeclareMathOperator{\Lip}{Lip}
 \DeclareMathOperator{\dist}{dist}
 \DeclareMathOperator{\esssup}{ess\, sup}

\newcommand{\sect}[1]{\section{#1}  \setcounter{equation}{0}  }

\newcommand{\norm}[2]{\left\|#1\right\|_{#2}}
\newcommand{\ds}{\displaystyle}

\newcommand{\Poly}{\Pi}
\newcommand{\Pn}{\Poly_n}

\newcommand{\E}{{\mathcal E}}

\newcommand{\A}{{\mathcal A}}
\newcommand{\B}{\mathcal B}
\newcommand{\NN}{{\mathcal N}}
\newcommand{\N}{\mathbb N}
\newcommand{\R}{\mathbb R}

\newcommand{\andd}{\quad\mbox{\rm and}\quad}
\newcommand{\ttau}{\widetilde\tau}
\newcommand{\ntau}{\mathcal{T}}
\newcommand{\qq}{\mathcal{Q}}

\newcommand{\pp}{\mathcal{F}}
\newcommand{\tpp}{\widetilde\pp}
\newcommand{\opp}{\overline\pp}

\newcommand{\J}{{\mathcal{J}}}
\newcommand\w{{\omega}}

\newcommand{\aj}{\alpha_j}

\newcommand{\mon}{\Delta^{(1)}}
\newcommand{\con}{\Delta^{(2)}}
\newcommand{\qmon}{\Delta^{(q)}}

\def\be  {\begin{equation}}
\def\ee  {\end{equation}}

\newcommand{\wS}{\widetilde S}

\newcommand{\ineq}[1]{{\rm(\ref{#1})}}

\newcommand{\ie}{{\em i.e., }}
\newcommand{\eg}{{\em e.g.}}
\newcommand{\st}{\;\; \big| \;\;}

 \newcommand{\ec}{\end{comment}}
\newcommand{\bc}{ \begin{comment}
 }

\newenvironment{comment}[2]
{\bgroup\vspace{7pt}
\begin{tabular}{|p{5in}|}
\hline \qquad \bf \footnotesize Comment -- to be deleted in the final version \\
\hline
\quad\sl\footnotesize #1#2} {\\ \hline \end{tabular}
\vspace{7pt}\indent\egroup}

\title{{\sc Interpolatory pointwise estimates for convex polynomial approximation}\thanks{{\it AMS classification:} 41A29, 41A10, 41A25   {\it Keywords
and phrases:} Convex approximation by polynomials, Degree of approximation, Jackson-type interpolatory estimates.}}

\author{K. A.  Kopotun\thanks{Department of Mathematics, University of
Manitoba, Winnipeg, Manitoba, R3T 2N2, Canada ({\tt kirill.kopotun@umanitoba.ca}).  Supported by NSERC of Canada Discovery Grant RGPIN 04215-15.} \and
D. Leviatan\thanks{Raymond and Beverly Sackler School of Mathematical
Sciences, Tel Aviv University, Tel Aviv 6139001, Israel ({\tt  leviatan@tauex.tau.ac.il}).}
\and
I. L. Petrova\thanks
{Faculty of Mechanics and Mathematics, Taras
Shevchenko National University of Kyiv, 01601 Kyiv, Ukraine.   ({\tt irynapetrova1411@gmail.com}).}
\and
I. A. Shevchuk\thanks
{Faculty of Mechanics and Mathematics, Taras
Shevchenko National University of Kyiv, 01601 Kyiv, Ukraine ({\tt shevchuk@univ.kiev.ua}).}
}

\begin{document}

\maketitle

\begin{abstract}
This paper deals with approximation of smooth convex functions $f$ on an interval by convex algebraic polynomials which interpolate $f$ and its derivatives at the endpoints of this interval. We call such estimates ``interpolatory''.
One important corollary of our main theorem is the following result on approximation of $f\in \Delta^{(2)}$, the set of  convex functions,  from
 $W^r$, the space of   functions on $[-1,1]$ for which $f^{(r-1)}$ is  absolutely continuous   and
$\|f^{(r)}\|_{\infty} := \esssup_{x\in[-1,1]} |f^{(r)}(x)|  < \infty$:
\begin{quote}
For any   $f\in W^r \cap\con$, $r\in\N$, there exists a number
$\NN=\NN(f,r)$,   such that for every $n\ge \NN$, there is an algebraic polynomial of degree $\le n$ which is in $\con$ and such that
\[
  \norm{ \frac{f-P_n}{\varphi^r} }{\infty} \leq \frac{c(r)}{n^r} \norm{f^{(r)}}{\infty} ,
\]
where $\varphi(x):= \sqrt{1-x^2}$.
\end{quote}
For $r=1$ and $r=2$, the above result holds with $\NN=1$ and is well known. For $r\ge 3$, it is not true, in general, with $\NN$ independent of $f$.
\end{abstract}

%

\sect{Introduction and main results}

We start by recalling some standard notation. As usual, $C^r(I)$ denotes the space of $r$ times continuously differentiable functions on a closed interval $I$,  $C^0(I):=C(I)$ is the space of continuous functions on $I$, equipped with the uniform norm which   will be denoted by $\norm{\cdot}{I}$. For $k\in\N$ and an interval $I$,
$\Delta^k_u(f,x;I):= \sum_{i=0}^k(-1)^i  \binom ki    f(x+(k/2-i)u)$ if $x\pm ku/2\in I$ and $:=0$, otherwise, and
$
\w_k(f,t;I):=\sup_{0<u\le t}\|\Delta^k_u(f,\cdot;I)\|_{I}
$
is the $k$th modulus of smoothness of $f$ on $I$.
When dealing with $I=[-1,1]$, we suppress referring to the interval and  use the notation  $\|\cdot\|:=\|\cdot\|_{[-1,1]}$,  $\omega_k(f,t):=\omega_k(f,t;[-1,1])$, $C^r := C^r[-1,1]$,  etc.
We denote by $\qmon$ the class of all $q$-monotone functions on $[-1,1]$, \ie continuous functions such that $\Delta^q_u(f,x) \ge 0$ for all $x\in[-1,1]$ and $u>0$. In particular,
  $\mon$ and $\con$ are the classes of all monotone and convex functions on $[-1,1]$, respectively.
Also,
\be\label{varphi}
\varphi(x):=\sqrt{1-x^2}\quad\text{and}\quad\rho_n(x):= \varphi(x)n^{-1}+ n^{-2}, \; n\in\N,
\ee
$\rho_0(x)\equiv 1$, and  $\Pn$ denotes  the space of algebraic polynomials of degree $\le n$.

The following classical Timan-Dzyadyk-Freud-Brudnyi direct theorem for the approximation by algebraic polynomials (see \eg{} \cite{DL}*{Theorem 8.5.3}) shows that
 the order of approximation becomes significantly better near the endpoints of $[-1,1]$:
if $k\in\N$, $r\in\N_0$ and $f\in C^r$, then for each $n\geq k+r-1$ there is a polynomial $P_n\in\Pn$ satisfying
\be \label{classdir}
|f(x)-P_n(x)| \le c(k,r) \rho_n^r(x) \w_k(f^{(r)}, \rho_n(x)) , \quad x\in [-1,1] .
\ee

Clearly, if we require that the approximating polynomials interpolate $f$ as well as its derivatives at the endpoints, and we are successful, then the estimates should become even better.

 Indeed, the following Telyakovskii-Gopengauz-type
(\ie ``interpolatory''-type) theorem is an immediate consequence of \cite{K}*{Corollary 2-3.4} (see \eg{}  \cite{K} for the history of this problem).

\begin{theorem}[\mbox{see \cite{K}*{Corollary 2-3.4}}] \label{thk-sim}
Let $r\in\N_0$, $k\in\N$ and $f\in  C^r$. Then for any $n\ge \max\{k+r-1, 2r+1\}$, there is a polynomial $P_n \in \Pn$ such that
\ineq{classdir} is valid and, moreover,
\be \label{sim2}
|f(x)-P_n(x)| \le c(r,k) \varphi^{2r}(x)  \w_k(f^{(r)}, \varphi^{2/k}(x) n^{-2(k-1)/k} ) , \quad \text{if }\; 1-n^{-2} \le |x| \le 1.
\ee
\end{theorem}

It follows from \cite{K}*{Theorem 3}  that, for any $\gamma \in\R$, the quantity  $\varphi^{2/k}(x) n^{-2(k-1)/k}$ in \ineq{sim2} cannot be replaced by $\varphi^{2\beta}(x) n^\gamma$ with $\beta > 1/k$.
Hence, the estimate \ineq{sim2} provides the optimal rate of approximation near the endpoints of $[-1,1]$.

It is a natural question if these estimates are valid if we approximate $q$-monotone functions by $q$-monotone polynomials.
Of course, as is rather well known,       \ineq{classdir} may not be valid in the $q$-monotone case for certain $r$ and $k$ even if $n$ is allowed to depend on the function $f$ that is being approximated.
For example, this is the case if
(i) $1\le q\le 3$, $0\le r\le q-1$ and  $r+k \ge q+2$ (\cite{LS98} if $q=1$, \cite{Yus} if $q=2$ or $q=3$), and
(ii) $q\ge 4$ and $r+k \ge 3$ (\cite{BP}).

Moreover, for any $q, r, k, n\in\N$, there exists a function $f_n\in C^r\cap\qmon$ such that   \ineq{sim2} is not valid for any polynomial $P_n\in\Pn\cap\qmon$ (the construction of such an $f_n$ is the same as in \cite{KLSconspline}, see also \cites{KLS-mon,GLSW,petr}). This means that, in the case $r\ge 1$, \ineq{sim2} cannot be true for all functions $f\in C^r\cap\qmon$  and all $n\ge \NN(k,r,q)$. We emphasize that this does NOT mean that, for each fixed $f\in C^r\cap\qmon$, \ineq{sim2} is invalid for sufficiently large $n$, \ie \ineq{sim2} may still be valid if $n\ge \NN(f)$ (in fact, the proof of this fact in the case $q=k=2$ is the main result of this paper).

  If $r=0$ and $k$ is ``small'', then the situation is  different:  for any $q,n\in\N$, if $r=0$ and $1\le k\le 2$, then   \ineq{classdir} and \ineq{sim2} are both valid for $q$-monotone approximation (it is possible to show that the case for $k=1$ follows from that for $k=2$).
      Indeed, the following interpolatory estimate follows from \cite{DY} ($q=1$), \cites{L, Yu} ($q=2$) and \cite{CG} ($q\ge 3$):
for any $q,n\in\N$ and $f\in C \cap\qmon$, there exists a polynomial $P_n\in\Pn\cap\qmon$ such that
\be \label{lineq}
|f(x)-P_n(x)|\le c(q)\omega_2\left(f,\varphi(x)/n\right),\quad x\in [-1,1],
\ee
where $c$ is an absolute constant. Additionally, \ineq{classdir} and \ineq{sim2} with $n\ge 2$ are valid for convex approximation (\ie $q=2$) if $r=0$ and $k=3$ (\cite{K}), and the case $q=3$, $r=0$ and $k=3$ or $k=4$ is still unresolved (in fact, it is not even known if \ineq{classdir} holds if $(q,r,k)=(3,0, 4)$).

Recently, we were able to show (see \cite{KLS-mon}) that \ineq{classdir} and \ineq{sim2} hold for monotone approximation ($q=1$) if $r\in\N$, $k=2$ and $n\ge\NN(f,r)$, and the main purpose of this paper is to prove an analogous result for convex approximation ($q=2$).
 In fact, we follow similar ideas and apply some of the construction in \cite{KLS-mon}, but there are some additional rather significant technical difficulties that we have to overcome in this case (for example, proofs in the cases for $r=1$ and $r\ge 2$ turn out to be completely different). Also, one of the important tools that we are using is our recent result \cite{KLSconspline} on convex approximation of $f\in C^r\cap\con$, by convex piecewise polynomials (see \thm{thm-convex-spline} below).

The following theorem is the main result in this manuscript.

\begin{theorem}\label{thm1} Given $r\in\N$,   there is a constant $c=c(r)$ with the property that if $f\in C^r\cap\con$, then there exists a number $\NN=\NN(f,r)$, depending on $f$ and $r$, such that for every $n\ge \NN$, there is  $P_n\in \Pn \cap \con$  satisfying
\be\label{interpol}
|f(x)-P_n(x)|\le c(r)\left( \varphi(x)/n \right)^r\omega_2\left(f^{(r)}, \varphi(x)/n\right),\quad x\in[-1,1] .
\ee
Moreover, for $x\in \left[-1, -1+ n^{-2}\right] \cup \left[1-n^{-2}, 1\right]$ the following stronger estimates are valid:
\be\label{interpol1}
|f(x)-P_n(x)|\le c(r)\varphi^{2r}(x)\omega_2\left(f^{(r)}, \varphi(x)/n\right)
\ee
and
\be\label{interpolaux}
|f(x)-P_n(x)|\le c(r)\varphi^{2r}(x) \w_1 \left(f^{(r)},  \varphi^2(x) \right) .
\ee
\end{theorem}

\begin{remark}
 \cite{KLSconspline}*{Theorem 2.3}  implies that \thm{thm1} is NOT valid with $\NN$ independent of $f$.
\end{remark}

We now discuss some corollaries and applications of \thm{thm1}.

Recall that, given a number $\alpha>0$,
  $\Lip^*\alpha $ denotes the class of all functions $f$ on $[-1,1]$ such that
$\w_2(f^{(\lceil\alpha\rceil - 1)}, t) =O\left(  t^{\alpha-\lceil\alpha\rceil+1}\right)$.
Together with the classical inverse theorems (see \eg{} \cite{KLPS}*{Theorem 5 and Corollary 6}), \ineq{classdir} implies that,  if $\alpha>0$, then a function $f$ is in $\Lip^*\alpha$ if and only if
 \be \label{lipin}
 \inf_{P_n\in\Pn} \norm{\rho_n^{-\alpha}  (f -P_n )}{} =  O(1) .
 \ee

\begin{corollary} \label{maincor}
If $\alpha>0$ and  $f \in \Lip^*\alpha\cap\con$, then there exists a constant $C=C(\alpha)$ such that, for all sufficiently large $n$, there are polynomials
$P_n\in \Pn \cap \con$  satisfying
\be \label{auxcor}
|f(x)-P_n(x)| \leq C  \left(  \varphi(x)/n\right)^\alpha, \quad x\in [-1,1].
\ee
\end{corollary}
For $0<\alpha<2$, \ineq{auxcor} follows from \ineq{lineq} (and was stated in \cite{L}).

In order to state another  consequence of \thm{thm1} we recall that $W^r$ denotes the space of   functions on $[-1,1]$ for which $f^{(r-1)}$ is  absolutely continuous   and
$\norm{f^{(r)}}{\infty} := \esssup_{x\in[-1,1]} |f^{(r)}(x)|  < \infty$.

\begin{corollary} \label{secondcor}
For any   $f\in W^r \cap\con$, $r\in\N$, there exists a number
$\NN=\NN(f,r)$,   such that 
\be \label{estim1}
\sup_{n \ge \NN} \; \inf_{P_n\in \Pn \cap \con} \norm{ \frac{f-P_n}{\varphi^r (\min\{1/n, \varphi \})^r} }{\infty} \leq  c(r)  \norm{f^{(r)}}{\infty}  .
\ee
In particular,
\be \label{est2}
\sup_{n \ge \NN} \; \inf_{P_n\in \Pn \cap \con} \norm{ \frac{f-P_n}{\varphi^r} }{\infty} \leq \frac{c(r)}{n^r} \norm{f^{(r)}}{\infty}  .
\ee
\end{corollary}

 It follows from \cite{KLSconspline}*{Theorem 2.3} that, if $r\ge 2$ and $r\ge 3$, then, respectively, inequalities \ineq{estim1} and \ineq{est2} are not true, in general, with $\NN$ independent of $f$.
 For all other $r\in\N$, these inequalities hold with $\NN=1$ which is a corollary of \ineq{lineq} with $q=2$.


\sect{Notations and some inequalities for the Chebyshev partition}

Most symbols used in this paper were introduced and discussed in \cite{KLS-mon}. For convenience, we list them
 in the following table which also includes symbols introduced in the previous section.
Note that, in the proofs below (but not in definitions and statements), we often omit writing index ``$n$'' if it does not create any confusion (thus, we write ``$\rho$'' instead of ``$\rho_n$'', ``$x_j$'' instead of ``$x_{j,n}$'', etc.).
 \\[2mm]
\setlength{\extrarowheight}{5.0pt}
 \begin{longtable}{| p{.15\textwidth} | p{.80\textwidth} |}
  \hline
 \multicolumn{2}{|c|}{\bf Chebyshev knots and Chebyshev partition}\\  \hline
$x_j := x_{j,n}$ &  $:= \cos(j\pi/n)$, $0\leq j \leq n$; $1$ for $j<0$ and $-1$  for $j>n$ (Chebyshev knots) \\
$T_n$ & $:=(x_{j})_{j=0}^n$ (Chebyshev partition) \\
$I_j:= I_{j,n}$ & $:= [x_{j},x_{j-1}]$\\
$h_j := h_{j,n}$ & $:= |I_{j,n}|  = x_{j-1}-x_j$\\
$I_{i,j}$ & $:=\bigcup_{k=\min\{i,j\}}^{\max\{i,j\}}I_k = \left[ x_{\max\{i,j\}}, x_{\min\{i,j\}-1} \right]$, $1\le i,j\le n$ (the smallest interval containing both $I_i$ and $I_j$)\\
$h_{i,j}$ & $:=|I_{i,j}|=\sum_{k=\min\{i,j\}}^{\max\{i,j\}}h_k = x_{\min\{i,j\}-1} - x_{\max\{i,j\}}$ \\
$\psi_j$ & $:=\psi_j(x) :=  |I_j|/(|x-x_j|+|I_j|)$ \\
$\varphi(x)$ & $:=\sqrt{1-x^2}$\\
$\rho_n(x)$ & $:= \varphi(x)n^{-1}+ n^{-2}$, $n\in\N$, and $\rho_0(x) \equiv 1$ \\
$\delta_n(x)$ & $:= \min\{ 1, n\varphi(x)\}$\\
  \hline
\multicolumn{2}{|c|}{\bf $k$-majorants}\\  \hline
 $\Phi^k$ & $:=\left\{ \psi\in C[0,\infty)  \;  \big| \;  \psi\uparrow, \;  \psi(0)=0, \; \text{and}\;   t_2^{-k} \psi(t_2) \leq t_1^{-k} \psi (t_1)
   \; \mbox{\rm for $0<t_1\leq t_2$} \right\}$.
 Note: if $f\in C^r$, then $\phi(t) := t^r \w_k(f^{(r)}, t)$  is equivalent to a function from $\Phi^{k+r}$
    \\
  \hline
 \multicolumn{2}{|c|}{\bf Piecewise polynomials on Chebyshev partition}\\  \hline
 $\Sigma _{k}:=\Sigma_{k,n}$ &    the set of   continuous   piecewise polynomials of degree $\leq k-1$ with knots at $x_j$, $1\leq j \leq n-1$\\
 $\Sigma^{(1)}_{k}:=\Sigma^{(1)}_{k,n}$ &    the set of   continuously differentiable piecewise polynomials of degree $\leq k-1$ with knots at $x_j$, $1\leq j \leq n-1$\\
$p_j :=  p_j(S)$ &  $:=  S|_{I_j}$, $1\le j\le n$ (polynomial piece of $S$ on the interval $I_j$)\\
$ b_{i,j}(S,\phi)$ & $\ds := 
\frac{\|p_i-p_j\|_{I_i}}{\phi(h_j)}\left(\frac{h_j}{h_{i,j}}\right)^k$,   where $\phi\in\Phi^k$, $\phi\not\equiv 0$  and $S\in\Sigma_k$ \\
$b_k(S,\phi, A)$ &  $\ds :=\max_{1\leq i,j\leq n} \left\{ b_{i,j}(S,\phi) \st I_i\subset  A\andd I_j\subset  A\right\}$, where an interval $A\subseteq [-1,1]$ contains at least one interval $I_\nu$\\
$b_k(S,\phi)$ &  $:=  b_{k}(S,\phi, [-1,1]) =  \max_{1\le i,j\le n}b_{i,j}(S,\phi)$  \\
 \hline
 \multicolumn{2}{|c|}{\bf Constants}\\  \hline
 $C(\gamma_1, \dots, \gamma_\mu)$ & positive constants depending only on parameters $\gamma_1, \dots, \gamma_\mu$ that may be different on different occurences\\
 $c$ &    positive constants that are either absolute or may only depend on the parameters $k$ and  $r$ (if present) \\
 $C_i$ & positive constants that are fixed throughout this paper\\
 $\overset{\gamma_1, \dots, \gamma_\mu}{\sim}$ &    $A\overset{\gamma_1, \dots, \gamma_\mu}{\sim} B$ iff  $C^{-1} B \le A \le C  B$, for some positive constant $C = C(\gamma_1, \dots, \gamma_\mu)$ \\
 \hline
 \multicolumn{2}{|c|}{\bf Indicator functions and truncated powers}\\  \hline
$ \chi_j(x)$ &  $:= \chi_{[x_j, 1]}(x) := 1$, if $x_j \leq x \leq 1$, and $:= 0$, otherwise\\
$\Phi_j(x)$ & $:=(x-x_j)_+  := (x-x_j)\chi_j(x) = \int_{-1}^x \chi_j(t) dt$  \\

\hline
\end{longtable}

\medskip



We now collect all   facts and inequalities for the Chebyshev partition that we need throughout this paper. Many of them are checked by straightforward calculations (also, see \eg{} \cites{DS,KLS-mon,S} for references).
Unless specified otherwise, it is assumed that $1\le j\le n$, $x,y\in [-1,1]$.

\begingroup
\allowdisplaybreaks
\begin{align}
\label{rho} & n^{-1} \varphi(x) <\rho_n(x)<h_{j} <5\rho_n(x),\quad x\in I_{j}  \\
\label{hj} &h_{j\pm 1} <3h_{j}    \\
\label{rho1} &\rho^2_n(y) <4\rho_n(x)(|x-y|+\rho_n(x))   \\
\label{rho11} &(|x-y|+\rho_n(x))/2 <(|x-y|+\rho_n(y))<2(|x-y|+\rho_n(x)) \\
 \label{newauxest} & \rho_n(x)  \le  |x-x_j|, \;  \text{for any }\; 0\leq j\leq n \andd  x\notin(x_{j+1}, x_{j-1})  \\
 \label{delta} & \delta_n(x)\le n\varphi(x)< \pi\delta_n(x), \;  \text{if}\quad x\in[-1,x_{n-1}]\cup[x_1,1], \andd \delta_n(x)=1, \;  \text{if}\quad x\in[x_{n-1},x_1] \\
 \label{anotherauxest}
& \rho_n^2(x)   < 8 h_j \left(|x-x_j|+\rho_n(x) \right)  \\
\label{auxsum}
& \left( \frac{\rho_n(x)}{\rho_n(x) + |x-x_j|} \right)^2     < c \psi_j(x) \\
\label{distance1}
& \rho_n(x) + |x-x_j| \sim \rho_n(x) + \dist(x, I_j) \\
\label{sumpsi}
& \sum_{j=1}^{n}\psi_j^{2}(x) \leq c  \\
\label{sumest}
& \sum_{j=1}^{n} \left( \frac{\rho_n(x)}{\rho_n(x) + \dist(x, I_j)} \right)^4 \leq c  \\
\label{delta1} 
& c\psi_j^{2}(x)\delta_n^2(x) \le \frac{1-x^2}{ (1+x_{j-1})(1-x_j)} \leq   c \delta_n^2(x) \psi_j^{-2}(x) \\
\label{hjrho}
& c\psi_j^{2}(x)\rho_n(x) \leq h_j \leq c\psi_j^{-1}(x)\rho_n(x) \\
  \label{estphi}
& c \psi_j^{2k}(x) \phi(\rho_n(x)) \le \phi(h_j)  \le  c \psi_j^{-k}(x) \phi(\rho_n(x))
\end{align}
\endgroup

\sect{Auxiliary results on polynomial approximation of indicator functions and truncated powers} \label{auxresults}

 Recall the notation
\be \label{deftj}
t_j(x) := \left( \frac{\cos 2n\arccos x}{x-x_j^0} \right)^2 + \left( \frac{\sin 2n\arccos x}{x-\bar x_j} \right)^2 ,
\ee
where
$\bar x_j := \cos((j-1/2)\pi/n)$ for $1\leq j \leq n$,
$x_j^0 := \cos((j-1/4)\pi/n)$ for $1\leq j <n/2$,
$x_j^0 := \cos((j-3/4)\pi/n)$ for $n/2\leq j \leq n$,
and note that $t_j\in\Poly_{4n-2}$ and, for all $1\leq j \leq n$,
\be \label{behtj}
 t_j(x) \sim (|x-x_j|+h_j)^{-2} , \quad x\in [-1,1],
\ee
 (see  \eg{} \cite{S} or \cite{K-sim}*{(22), Proposition 5}).

For $\gamma_1, \gamma_2\in\N_0$,  $\xi,\mu\in\N$, and $1\le j\le n$, we let
\[
\ntau_j(x) :=  \ntau_{j,n}(x) := \ntau_{j,n}(x; \gamma_1, \gamma_2,\xi,\mu) :=   d_j^{-1} \int_{-1}^x (y-x_j)^{\gamma_1}(x_{j-1}-y)^{\gamma_2}(1-y^2)^\xi t_j^\mu(y)\, dy ,
\]
where $d_j := d_j (\gamma_1, \gamma_2,\xi,\mu)$ is the normalizing constant such that $\ntau_j(1)=1$. Then, it is possible to show (see \eg{} \cite{K-coconvex}*{Proposition 4}) that, for sufficiently large $\mu$, function $\ntau_j$ is well defined and is a polynomial of degree $\le c \mu n$ (with some absolute constant $c$), and
\[
  d_j \sim (1+x_{j-1})^\xi (1-x_j)^\xi h_j^{-2\mu+1+\gamma_1 +\gamma_2} .
\]
Also,
\be \label{intineq}
1-x_{j-1} < \int_{-1}^1 \ntau_j(t)dt < 1-x_j , \quad 1\le j \le n .
\ee
Indeed, denoting for convenience $\vartheta(y) :=(y-x_j)^{\gamma_1}(x_{j-1}-y)^{\gamma_2}(1-y^2)^\xi t_j^\mu(y)$, we have
\begin{eqnarray*}
\lefteqn{ \int_{-1}^1 \ntau_j(t)dt < 1-x_j  \iff  \int_{-1}^1 \int_{-1}^t \vartheta(t) dt < (1-x_j) \int_{-1}^1 \vartheta(t) dt} \\
 &  \iff &
 \int_{-1}^1 (t-x_j) \vartheta(t) dt   =  d_j(\gamma_1+1, \gamma_2, \xi, \mu) > 0
\end{eqnarray*}
(the other inequality is proved similarly).

Now, for the polynomials
\[
\tau_j(x) := \ntau_{j,n}(x; 0,0,\xi,\mu) \andd \ttau_j(x) := \ntau_{j,n}(x; 1,1,\xi,\mu) ,
\]
  the following lemma was proved in \cite{KLS-mon}.


\begin{lemma}[\cite{KLS-mon}*{Lemmas 4.1 and 4.2}]\label{lem5}
If $\alpha,\beta \geq 1$, then for sufficiently large $\xi$ and $\mu$ depending only on $\alpha$ and $\beta$ and for each
 $1\leq j\leq n-1$,
 polynomials
$\tau_j$ and $\ttau_j$of degree $\le C(\alpha,\beta)n$ satisfy
\be \label{taudoubleprime}
\tau_j'(x)\ge C(\alpha,\beta) |I_j|^{-1}\delta_n^{8\alpha}(x)\psi_j^{30(\alpha+\beta)}(x), \quad x\in [-1,1] ,
\ee
\be \label{decreasing}
\ttau_j'(x) \leq 0, \quad \mbox{\rm for }\; x\in [-1, x_j] \cup [x_{j-1}, 1] ,
\ee
and, for all $x\in[-1,1]$,
\be \label{derivatives}
\max\left\{ \left|\tau_j'(x)\right|, \left|\ttau'_j(x) \right| \right\} \leq C(\alpha,\beta) |I_j|^{-1}\delta_n^\alpha(x)\psi_j^\beta(x)
\ee
and
\be\label{tauj}
\max\left\{ |\chi_j(x)-\tau_j(x)|,  |\chi_j(x)-\ttau_j(x)| \right\}     \le
C(\alpha,\beta)\delta_n^\alpha(x) \psi_j^\beta(x) .
\ee
\end{lemma}

\begin{remark}
The statement of this lemma is not valid if $j=n$ since $\chi_n\equiv 1$, $\tau_n(-1)=0$ and $\delta_n(-1)=0$.
\end{remark}

  Inequalities \ineq{intineq} imply that, for each $1\le j \le n-1$, there exists a constant $0<\lambda_j<1$ such that the polynomial
\be \label{qq}
\qq_j(x) := \qq_{j,n}(x) := \qq_{j,n}(x; \gamma_1, \gamma_2,\xi,\mu) := \int_{-1}^x \left( \lambda_j \ntau_j(t) + (1-\lambda_j) \ntau_{j+1} (t)\right) dt
\ee
satisfies $\qq_j(1)=1-x_j$. This implies that, if $\ntau_j$ is such that \ineq{tauj} is satisfied, then $\qq_j$ provides a ``good'' approximation of $\Phi_j$, $1\le j \le n-2$. The proof of this fact is rather standard.
Indeed, first note that, for $1\le j\le n-2$ and $x\in[-1,1]$,
\[
|\chi_j(x)-\chi_{j+1}(x)| \le |\chi_{I_{j+1}}(x)| \le C \delta_n^\alpha(x) \psi_j^\beta(x) \andd \psi_j(x) \sim \psi_{j+1}(x).
\]
Now, if $x\le x_j$, then (assume that $\beta >1$)
\begin{eqnarray*}
|\Phi_j(x)-\qq_j(x)| &\le&  \left|\int_{-1}^x \left( \lambda_j |\ntau_j(t)-\chi_j(t)| + (1-\lambda_j) |\ntau_{j+1} (t)-\chi_j(t)|\right)  dt \right| \\
&\le & C \int_{-1}^x \delta_n^\alpha(t) \psi_j^\beta(t) dt \le C \delta_n^\alpha(x) \int_{-\infty}^x   |I_j|^\beta (x_j-t+|I_j|)^{-\beta}  dt \\
&\le & C |I_j| \delta_n^\alpha(x)\psi_j^{\beta-1} (x)
\end{eqnarray*}
and, if $x>x_j$, then, similarly,
\begin{eqnarray*}
|\Phi_j(x)-\qq_j(x)|  & = &  \left|\int_x^1 \left( \chi_j(t)-\qq_j'(t)\right) \, dt  \right|          \\
&\le&  \left|\int_{x}^1 \left( \lambda_j |\ntau_j(t)-\chi_j(t)| + (1-\lambda_j) |\ntau_{j+1} (t)-\chi_j(t)|\right)  dt \right| \\
&\le & C |I_j| \delta_n^\alpha(x)\psi_j^{\beta-1} (x) .
\end{eqnarray*}

Now, for $1\le j\le n-1$, defining
\be \label{pp}
\pp_j (x):= \pp_{j,n}  (x):= \qq_{2j,2n} (x; 0,0,\xi, \mu) \andd \tpp_j (x):= \tpp_{j,n} (x):= \qq_{2j-1,2n}(x; 1,1,\xi, \mu) ,
\ee
and noting that $x_{j,n} = x_{2j, 2n}$, $h_{j,n} \sim h_{2j,2n}\sim h_{2j-1,2n}$, $\psi_{j,n} \sim \psi_{2j, 2n} \sim \psi_{2j-1, 2n}$, $\delta_{n} \sim \delta_{2n}$,
we   have the following result which follows from \lem{lem5}.

\begin{lemma} \label{newlem5}
If $\alpha,\beta \geq 1$, then for sufficiently large $\xi$ and $\mu$ depending only on $\alpha$ and $\beta$ and for each
$1\leq j\leq n-1$, 
 polynomials
$\pp_j$ and $\tpp_j$ of degree $\le C(\alpha,\beta) n$ defined in \ineq{pp} satisfy
\be \label{newtaudoubleprime}
\pp_j''(x)\ge C(\alpha,\beta) |I_j|^{-1}\delta_n^{8\alpha}(x)\psi_j^{30(\alpha+\beta)}(x), \quad x\in [-1,1] ,
\ee
\be \label{newdecreasing}
\tpp_j''(x) \leq 0, \quad \mbox{\rm for }\;   x\in [-1, x_{j}] \cup [x_{j-1}, 1] ,                               
\ee
and, for all $x\in[-1,1]$,
\be \label{newderivatives}
\max\left\{ \left|\pp_j''(x)\right|, \left|\tpp_j''(x) \right| \right\} \leq C(\alpha,\beta) |I_j|^{-1}\delta_n^\alpha(x)\psi_j^\beta(x)
\ee

\be\label{newtauj}
\max\left\{ |\chi_j(x)-\pp_j'(x)|,  |\chi_j(x)-\tpp_j'(x)| \right\}     \le
C(\alpha,\beta) \delta_n^\alpha(x) \psi_j^\beta(x)
\ee
and
\be\label{secder}
\max\left\{ |\Phi_j(x)-\pp_j(x)|,  |\Phi_j(x)-\tpp_j (x)| \right\}     \le
C(\alpha,\beta) |I_j| \delta_n^\alpha(x) \psi_j^{\beta-1}(x) .
\ee
\end{lemma}


\sect{Auxiliary results on properties of piecewise polynomials}

\begin{lemma}[\cite{KLS-mon}*{Lemma 5.1}]   \label{b_k<c}
Let $k\in\N$, $\phi\in\Phi^k$, $f\in C[-1,1]$ and $S\in\Sigma_{k,n}$. If
$
\omega_k(f,t)\le\phi(t)
$
and
$ 
|f(x)-S(x)|\le \phi(\rho_n(x))$, $x\in[-1,1]$,
then
$b_k(S, \phi)\le c(k)$.
\end{lemma}

\begin{lemma}[\cite{DLS}*{Lemma 2.1}]
\label{important}
Let   $k\ge 3$,  $\phi\in\Phi^k$ and $S\in\Sigma_{k,n}^{(1)}$. Then
$
b_k(S,\phi)\le\,c(k) \norm{  \rho^2_n \phi^{-1}(\rho_n) S'' }{\infty}.
$
\end{lemma}

The following lemma on simultaneous polynomial approximation of piecewise polynomials and their derivatives
is an immediate corollary of \cite{KLS-mon}*{Lemma 8.1} (with $q=r=2$ and $k\ge 2$).

\begin{lemma}[\cite{KLS-mon}*{Lemma 8.1}]    \label{uncon}
Let $\gamma >0$, $k\in\N$, $\phi\in\Phi^k$, and let $n, n_1\in\N$ be such that $n_1$ is divisible by $n$. If
 $S\in\Sigma_{k,n}$, then there exists a polynomial $D_{n_1}(\cdot, S)$ of degree $\leq Cn_1$ such that
\be \label{7.23'}
\left|S(x)-D_{n_1}(x, S)\right|\le
C\delta_n^{\gamma}(x)  \phi(\rho_n(x))   b_k(S,\phi).
\ee
Moreover, if   $S\in  C^{1}$  and $A := [x_{\mu^*}, x_{\mu_*}]$, $0\leq \mu_* < \mu^* \leq n$, then for all $x\in A\setminus \{x_j\}_{j=1}^{n-1}$,
 we have
 \begin{eqnarray} \label{7.24'}
  \left|S''(x)-D_{n_1}''(x,S)\right|
 &\le &  C \delta_n^{\gamma}(x)
\frac { \phi(\rho_n(x))}{\rho_n^2(x)} \bigg(  b_k(S,\phi,A) +   b_k(S,\phi) \times     \\ \nonumber
 &&
\times   \frac{n}{n_1} \left(\frac{\rho_n(x)}{ \dist(x,[-1,1] \setminus A)}\right)^{\gamma+1}   \bigg).
\end{eqnarray}
All constants $C$  may depend only on $k$ and $\gamma$ and are independent of the ratio $n_1/n$.
\end{lemma}

\sect{Convex polynomial approximation of   piecewise polynomials with  ``small'' derivatives}

\begin{lemma}\label{step11}
Let $\alpha>0$, $k\in\N$ 
and $\phi\in\Phi^k$,
be given. If $S\in\Sigma_{k,n} \cap\con$    is such that
\be\label{approx21}
|S''(x)|\le
\frac{\phi(\rho_n(x))}{\rho_n^2(x)},\quad x\in [x_{n-1},x_1]\setminus
\{x_j\}_{j=1}^{n-1},
\ee
 \be \label{jumps1}
 0\leq S'(x_j+) - S'(x_j-) \leq \frac{\phi(\rho_n(x_j))}{\rho_n(x_j)}, \quad 1\leq j \leq n-1 ,
 \ee
and
\be\label{approx31}
S''(x)=0, \quad x\in [-1, x_{n-1}) \cup (x_1, 1],
\ee
then there is a polynomial $P \in\con\cap\Poly_{Cn}$, $C=C(k,\alpha)$,    such that
\be\label{approx1}
|S(x)-P (x)|\le C(k,\alpha) \delta_n^\alpha(x) \,\phi\left(\rho_n(x)\right),\quad x\in [-1,1].
\ee
\end{lemma}


\begin{proof}
Denote by $S_1$ the  piecewise linear continuous function interpolating $S$ at the points $x_j$, $0\le j\le n$, and let $l_j:= S_1\big|_{I_j}$. Then $S_1\in \con$,         
\be\label{approx51}
S_1(x)=S(x),\quad x\in I_1\cup I_n,
\ee
and, for $x\in I_j$, $1\le j \le n$, we have by Whitney's inequality and \ineq{rho}
\[
|S(x)-S_1(x)| \le c \w_2(S, h_j; I_j) \le c h_j^2 \|S''\|_{L_\infty(I_j)} \le c \phi(h_j) ,
\]
which can be rewritten as
\be\label{approx41}
|S(x)-S_1(x)|\le c\phi(\rho_n(x)),\quad x\in[-1,1].
\ee
We now write $S_1$ as
\[
S_1(x) = S_1(-1)+ S_1'(-1)(x+1)+\sum_{j=1}^{n-1} \aj \Phi_j(x) , \quad \aj := S_1'(x_j+)-S_1'(x_j-) ,
\]
note that, by Markov and Whitney inequalities,
\begin{eqnarray*}
0\le \aj &=& l_j'(x_j)- l_{j+1}'(x_j) \le c h_j^{-1} \norm{l_j-l_{j+1}}{I_{j}\cup I_{j+1}} \le c h_j^{-1} \w_2(S, h_j;  I_{j}\cup I_{j+1}) \\
& \le &   c h_j \left( \norm{S''}{L_\infty(I_{j})} + \norm{S''}{L_\infty(I_{j+1})} \right)     + c \left(  S'(x_j+) - S'(x_j-) \right)    \\
& \le &   c h_j^{-1} \phi(h_j) , \quad 1\le j \le n-1 .
\end{eqnarray*}
Now, if
\[
P(x) := S_1(-1)+ S_1'(-1)(x+1)+\sum_{j=1}^{n-1} \aj \pp_j(x) ,
\]
then $P$ is a convex polynomial of degree $\le Cn$ and,
in view of \ineq{approx51} and \ineq{approx41}, we only need to estimate $|S_1(x)-P(x)|$.
Note that \ineq{hjrho} implies, for all $1\leq j\leq n$ and $x\in [-1,1]$,
\[
\phi(h_j) \leq \phi\left( c\psi_j^{-1}(x)\rho_n(x) \right) \leq C\,\psi_j^{-k}(x)  \phi\left(  \rho_n(x) \right).
\]
Hence,
by \lem{newlem5}   and \ineq{sumpsi}, we have
\begin{eqnarray*}
|S_1(x)-P(x)|&\le & \sum_{j=1}^{n-1}\aj |\Phi_j(x)-\pp_j(x)| \le C \sum_{j=1}^{n-1} \phi(h_j) \delta_n^{\alpha}(x) \psi_j^{\beta}(x) \\
& \leq &
C \delta_n^{\alpha}(x) \phi \left(  \rho_n(x) \right)  \sum_{j=1}^{n-1}\psi_j^{\beta-k}(x)
  \leq
C \delta_n^{\alpha}(x) \phi \left(  \rho_n(x) \right)   ,
\end{eqnarray*}
provided $\beta \ge k+2$.
\end{proof}

\sect{One particular polynomial with controlled second derivative} \label{yetanother}

All constants $C$ in this section may depend on $k$, $\alpha$ and $\beta$.

We start with the following auxiliary lemma.

\begin{lemma}[\cite{LS2002}*{Lemma 9}] \label{lemcom}
Let $A := \{j_0, \dots, j_0+l_0\}$ and let $A_1,A_2 \subset A$ be such that $\#A_1=2l_1$ and $\#A_2=l_2$. Then, there exist $2l_1$ constants $a_i$, $i\in A_1$, such that $|a_i|\le (l_0/l_1)^2$ and
\[
\frac{1}{l_2} \sum_{j\in A_2} (x-x_j) + \frac{1}{l_1} \sum_{j\in A_1} a_j (x-x_j) \equiv 0 .
\]
\end{lemma}

 \begin{lemma}\label{QM}
Let $\alpha >0$, $k\in\N$, $k\ge 2$, $\beta>0$ be sufficiently large ($\beta \ge k+7$ will do) and let $\phi\in\Phi^k$ be of the form $\phi(t):=t \psi(t)$, $\psi\in\Phi^{k-1}$.
Also,
let $E\subset [-1,1]$ be a closed interval which is the union of $m_E \ge 100$ of the
intervals $I_j$, and let a set $J\subset  E$ consist of $m_J$   intervals $I_j$, where
$1\le m_J < m_E/4$. Then there exists a
polynomial $Q_n(x)=Q_n(x,E,J)$ of degree $\le Cn$, satisfying
\begin{align}  \label{Qu}
Q_n''(x)&\ge C
\frac {m_E}{m_J} \delta_n^{\alpha_1}(x) \frac{\phi(\rho_n(x))}{\rho_n^2(x)}
\left(\frac{\rho_n(x)}{\max\{\rho_n(x),\dist (x,E)\}}\right)^{\beta_1},       \quad x\in J\cup([-1,1]\setminus E),
\end{align}
\be\label{Qu1}
Q_n''(x)\ge -\delta_n^{\alpha}(x) \frac{\phi(\rho_n(x))}{\rho_n^2(x)},\quad x\in E\setminus J,
\ee
and
\be   \label{Qu2}
|Q_n(x)|  \le    C \,m_E^{k_1}\delta_n^{\alpha}(x) \rho_n(x)\,\phi(\rho_n(x))\sum_{j: \, I_j\subset  E}\frac{h_j}{(|x-x_j|+\rho_n(x))^2},  \qquad   x\in[-1,1],
\ee
where $\alpha_1=8\alpha$,  $\beta_1= 60(\alpha+\beta)+k+1$ and $k_1=k+6$.
\end{lemma}

\begin{proof}
As in the proof of \cite{KLS-mon}*{Lemma 9.1}, we  may assume
 that $I_n \not\subset E$ provided that   the condition $m_J < m_E/4$ is replaced by   $m_J \leq m_E/4$.
Also, we use the same notation that was used in \cite{KLS-mon}:
 $\rho := \rho_n(x)$, $\delta := \delta_n(x)$, $\psi_j := \psi_j(x)$,
\begin{eqnarray*}
\E &:=&  \left\{ 1\leq j \leq n \st I_j \subset E\right\} , \quad
\J  :=   \left\{ 1\leq j \leq n \st I_j \subset J\right\},\\
j_* &:=&  \min \left\{ j \st j\in \E\right\}, \quad  j^*  :=   \max \left\{ j \st j\in \E\right\} , \\
\A &:= & \J \cup \{j_*, j^*\} \andd \B := \E \setminus \A .
\end{eqnarray*}
 %
Now, let
  $\widetilde E\subset E$ be the subinterval of $E$ such that
\begin{itemize}
\item[(i)] $\widetilde E$ is a union of $\lfloor m_E/3 \rfloor$ intervals $I_j$, and  
\item[(ii)]  $\widetilde E$ is  centered at $0$ as much as $E$   allows it, \ie among all subintervals of $E$ consisting of $\lfloor m_E/3 \rfloor$ intervals $I_j$, the center of $\widetilde E$ is closest to $0$.
\end{itemize}
Then (see \cite{KLS-mon}),
\be\label{compare}
{\rm if}\quad I_j \subset \widetilde E\quad{\rm and}\quad I_i \subset E\setminus \widetilde E,\quad{\rm then}\quad |I_j|\geq |I_i|,
\ee
\be \label{middle}
|I_j| \sim   \frac{|\widetilde E|}{m_E}, \quad \mbox{\rm for all }\; I_j \subset \widetilde E ,
\ee
and, with $\widetilde \E  :=   \left\{ 1\leq j \leq n \st I_j \subset \widetilde E\right\}$ and $\widetilde \B := \B \cap \widetilde\E = \widetilde\E \setminus\A$,
\be \label{numbers}
\# \widetilde \B \ge m_E/20.
\ee
Note that index $j=n$ is in none of  the sets $\A$, $\B$ and $\widetilde \B$.

It follows from \lem{lemcom} ($l_0 \sim m_E$, $l_1 \sim \lfloor \#\widetilde\B/2\rfloor \sim m_E$, $l_2 \sim m_J$)   that there exist   constants $\lambda_i$, $i\in \widetilde\B$, such that
$|\lambda_i| \le c$, $i\in \widetilde\B$,
and
\be \label{comb}
\frac{m_E}{m_J} \sum_{j\in\A} (x-x_j) +   \sum_{j\in\widetilde\B} \lambda_i (x-x_j) \equiv 0.
\ee
We now let $i_*$ be such that $I_{i^*}$ is the largest interval in $\widetilde E$ and $h_* := h_{i^*}  = |I_{i^*}|$, and
\[
Q_n(x) := \kappa \frac{\phi(h_*)}{h_*}  \left( \frac{m_E}{m_J} \sum_{j\in\A}   \pp_j(x)   +   \sum_{j\in\widetilde \B} \lambda_j  \opp_j(x) \right) ,
\]
where $\kappa$ is a sufficiently small absolute constant to be prescribed  and
\[
\opp_j :=
\begin{cases}
\tpp_j , & \mbox{\rm if}\; \lambda_j <  0 , \\
\pp_j, & \mbox{\rm if}\; \lambda_j \ge   0 .
\end{cases}
\]
 It follows from \ineq{compare}    that
 \[
 h_j \le h_*, \quad j\in  \E ,
 \]
 and so $\rho \le h_*$ and $ \phi(\rho)/\rho = \psi(\rho) \le \psi(h_*) =   \phi(h_*)/h_*$, for all $x\in E$ as well as all $x\not\in E$ such that  $h_* \ge  \rho$. If $x\not\in E$ and $h_* < \rho$, then by \ineq{rho}, \ineq{rho1} and \ineq{rho11}
 \begin{align*}
 \frac{\phi(h_*)}{h_*}  & \ge \frac{\phi(\rho) h_*^{k-1}}{\rho^k} \ge \frac{\phi(\rho)}{\rho^k} \max\{ h_{j_*}^{k-1}, h_{j^*}^{k-1}\}
 \ge c \frac{\phi(\rho)}{\rho^k} \cdot \frac{\rho^{2k-2}}{ (\min\{|x-x_{j_*}|, |x-x_{j^*}|\} + \rho )^{k-1}} \\
 & \ge c \phi(\rho) \frac{\rho^{k-2}}{\left(\max\{\rho,\dist (x,E)\}\right)^{k-1}} .
 \end{align*}
 Hence,
 \be \label{lest}
 \frac{\phi(h_*)}{h_*}  \ge c \frac{\phi(\rho)}{\rho}  \left( \frac{\rho}{\max\{\rho,\dist (x,E)\}} \right)^{k-1} , \quad \mbox{\rm for all}\; x\in [-1,1] .
 \ee

We now note that $\lambda_i \opp_j''(x)\ge 0$ if $j\in\B$ and $x\in J\cup([-1,1]\setminus E)$ (as well as for any $x\in I_{j_*} \cup I_{j^*}$).
Hence,   for these $x$, using \lem{newlem5},   \ineq{hjrho}, \ineq{auxsum} and \ineq{lest}   we have
\begin{eqnarray*}
Q_n''(x) & \geq & \kappa  \frac{\phi(h_*)}{h_*} \cdot  \frac{m_E}{m_J} \sum_{j\in\A}    \pp_j''(x)  \\
& \geq &
C \kappa \delta^{8\alpha}(x) \frac{\phi(h_*)}{h_*} \cdot \frac{m_E}{m_J} \sum_{j\in\A}    h_j^{-1}    \psi_j^{30(\alpha+\beta)}  \\
& \geq &
C \kappa \delta^{8\alpha}(x) \frac{\phi(h_*)}{\rho h_*} \cdot \frac{m_E}{m_J}    \sum_{j\in\A} \psi_j^{30(\alpha+\beta)+1}  \\
&\geq &
C \kappa \delta^{8\alpha}(x) \frac{\phi(h_*)}{\rho h_*} \cdot \frac{m_E}{m_J}   \sum_{j\in\A} \left( \frac{\rho}{\rho+|x-x_j|} \right)^{60(\alpha+\beta)+2} \\
& \geq &
C \kappa \delta^{8\alpha}(x) \frac{m_E}{m_J} \cdot  \frac{\phi(\rho)}{\rho^2} \left(\frac{\rho }{\max\{\rho ,\dist (x,E)\}}\right)^{60(\alpha+\beta)+k+1} ,
\end{eqnarray*}
 since, for $x\not\in E$, $\max\{\rho, \dist(x, E)\} \sim \min \left\{ |x-x_{j^*}|,|x-x_{j_*}|\right\} + \rho$, and, if $x\in J$, then $x\in I_j$ for some $j\in \A$,   so that
 $\rho/(|x-x_j|+\rho) \sim 1$ for that $j$.

If $x\in E\setminus J$ and $x\not\in I_{j_*} \cup I_{j^*}$,  then  there exists $j_0\in \B$ such that $x\in I_{j_0}$. If $j_0 \not\in\widetilde\B$, or if $j_0  \in\widetilde\B$ and $\lambda_{j_0} \ge 0$,  then, clearly, $Q_n''(x) \ge 0$.
 Otherwise, since $h_* \sim h_{j_0}$ by \ineq{middle}, we have using \ineq{newderivatives}
\begin{eqnarray*}
Q_n''(x) & \geq &   \kappa \lambda_{j_0} \phi(h_{*}) h_{*}^{-1} \tpp_{j_0}''(x)
  \geq   - C \kappa \phi(h_{j_0})  h_{j_0}^{-2} \delta^\alpha \psi_{j_0}^\beta  \\
  & \geq &
  - C \kappa  \frac{\phi(\rho)}{\rho^2} \delta^\alpha \geq - \frac{\phi(\rho)}{\rho^2} \delta^\alpha ,
\end{eqnarray*}
for sufficiently small $\kappa$.

We now estimate $|Q_n(x)|$.
Let
\[
L(x) := \kappa \frac{\phi(h_*)}{h_*}  \left( \frac{m_E}{m_J} \sum_{j\in\A}   \Phi_j(x)   +   \sum_{j\in\widetilde \B} \lambda_j  \Phi_j(x) \right) ,
\]

It follows from \cite{KLS-mon}*{(9.8)} that, for any $j\in\E$, $c m_E \le |E|/h_j \le m_E^2$. This implies that $h_* \le c |E|/m_E \le c m_E h_j$, $j\in\E$, and so
$\phi(h_*) \leq c m_E^k \phi(h_j)$, $j\in\E$.
Hence, using  \ineq{secder} as well as the estimate (see \cite{KLS-mon}*{pp. 1282-1283})
\[
\sum_{j\in\E} \phi(h_j) \psi_j^{\beta-1} \le C \phi(\rho) \sum_{j\in\E}   \frac{h_j \rho}{(|x-x_j|+\rho)^2}
\]
which is true if $\beta \ge k+7$,  we have  
\begin{eqnarray*}
\lefteqn{|Q_n(x)-L(x)|}\\
 & =& \kappa \frac{\phi(h_*)}{h_*}  \left| \frac{m_E}{m_J} \sum_{j\in\A}   \left( \pp_j(x) -\Phi_j(x) \right)   +   \sum_{j\in\widetilde\B} \lambda_j   \left( \opp_j(x) -\Phi_j(x) \right)  \right| \\
& \leq & C m_E \delta^\alpha \frac{\phi(h_*)}{h_*} \sum_{j\in\E} h_j  \psi_j^{\beta-1}
     \le C m_E^{k+1}  \delta^\alpha  \sum_{j\in\E} \phi(h_j) \psi_j^{\beta-1}       \\
 &\leq&   C m_E^{k+1} \delta^\alpha \phi(\rho) \sum_{j\in\E}   \frac{h_j \rho}{(|x-x_j|+\rho)^2}.
\end{eqnarray*}
It remains to estimate $|L(x)|$. First assume that $x\not\in E$. If $x\le x_{j^*}$, then $\Phi_j(x)=0$, $j\in\A \cup\widetilde \B$, and $L(x)=0$. If, on the other hand, $x>x_{j_*}$, then $\Phi_j(x)=x-x_j$, $j\in\A \cup\widetilde \B$, so that \ineq{comb} implies that $L(x)= 0$.  Hence, in particular, $L(x)=0$ for $x\in I_1\cup I_n$.

Suppose now that  $x\in E\setminus I_1$ (recall that we already assumed that $E$ does not contain $I_n$).
Then, as above,
   $h_* \leq c |E|/m_E \leq c \rho m_E$  and so
$\phi(h_*) \leq c m_E^k \phi(\rho)$.
 Also, $h_* \ge |E|/m_E^2$.
Hence, since $\delta = 1$ on $[x_{n-1}, x_1]$,
\begin{eqnarray*}
|L(x)|& \leq& C \frac{\phi(h_*)}{h_*}  \left(\frac{m_E}{m_J} \sum_{j\in\A} |x-x_j| +c  \sum_{j\in\widetilde \B} |x-x_j| \right)
  \le   C m_E^{k+3} \frac{\phi(\rho)}{|E|}   \sum_{j\in\E} |x-x_j| \\
  &\le & C m_E^{k+3} \frac{\phi(\rho)}{|E|}   \sum_{j\in\E} |E|
 \le  C m_E^{k+4} \delta^\alpha \phi(\rho) .
\end{eqnarray*}

It remains to note that
\[
1 = |E| \sum_{j\in\E} \frac{h_j}{|E|^2} \leq c |E| \sum_{j\in\E} \frac{h_j}{(|x-x_j|+\rho)^2} \leq c m_E^2 \sum_{j\in\E} \frac{\rho h_j}{(|x-x_j|+\rho)^2},
\]
and the proof is complete.
\end{proof}

\sect{Convex  polynomial approximation of piecewise polynomials}\label{sec5}

\begin{lemma}[\cite{DLS}*{Lemma 4.3}] \label{newlemmasec101}
Let $k\ge 3$, $\phi\in\Phi^k$ and  $S\in\Sigma_{k,n}^{(1)}$ be such that
$b_k(S,\phi) \leq 1$.
 If   $1\leq \mu, \nu \leq n$ are such that the interval $I_{\mu,\nu}$ contains at least $2k-5$  intervals $I_i$ and points $x_i^*\in (x_i, x_{i-1})$ so that
\[ 
\rho_n^2(x_i^*) \phi^{-1} (\rho_n(x_i^*)) |S''(x_i^*)| \leq 1,
\]
then, for every $1\leq j\leq n$, we have
\[ 
\norm{ \rho_n^2 \phi^{-1}(\rho_n) S''}{L_\infty(I_j)} \leq c(k) \left[ (j-\mu)^{4k} + (j-\nu)^{4k} \right] .
\]
\end{lemma}

\begin{theorem}\label{step1111} Let  $k,r\in\N$, $r\ge 2$, $k\geq r+1$,    and let $\phi\in\Phi^k$ be of the form $\phi(t):=t^r\psi(t)$, $\psi\in\Phi^{k-r}$. Also, let  $d_+\ge0$, $d_-\ge0$ and $\alpha\ge0$  be given. Then there is a number $\NN=\NN(k,r,\phi,d_+,d_-,\alpha)$  satisfying the following assertion. If  $n\ge \NN$ and $S\in\Sigma_{k,n}^{(1)}  \cap\con$ is such that
\be\label{d11}
b_k(S,\phi) \le 1,
\ee
and, additionally,
\begin{eqnarray}
\label{d+11}
\text{if $d_+>0$, then }  & &  d_+|I_2|^{r-2}\le\min_{x\in I_2}S''(x), \\
\label{d+is01}
\text{if $d_+=0$, then }  & &  S^{(i)}(1) =0, \; \text{for all }\;  2\leq i \leq k-2,    \\
\label{d-1}
\text{if $d_->0$, then }  & & d_-|I_{n-1}|^{r-2}\le\min_{x\in I_{n-1}}S''(x), \\
\label{d-is01}
\text{if $d_-=0$, then }  & &  S^{(i)}(-1) =0, \; \text{for all }\;  2\leq i \leq k-2,
\end{eqnarray}
 then
 there exists  a polynomial
$P\in\con\cap\Poly_{Cn}$, $C=C(k,\alpha)$,  satisfying, for all $x\in[-1,1]$,
\begin{eqnarray}
\label{approx101}
 & |S(x)-P(x)|\le C(k,\alpha)\,\delta_n^{\alpha}(x)\phi(\rho_n(x)),   &  \text{if $d_+>0$ and $d_->0$,}   \\
 \label{newapprox102}
 & |S(x)-P(x)|\le C(k,\alpha)\,  \delta_n^{\min\{\alpha, 2k-2\}} (x)   \phi(\rho_n(x)),   &  \text{if $\min\{d_+, d_-\} = 0$.}
\end{eqnarray}
\end{theorem}

The proof of \thm{step1111} is quite long and technical and is similar (with some rather significant changes) to that of \cite{KLS-mon}*{Theorem 10.2}. It is given in the last section of this paper.

\sect{Convex approximation by smooth piecewise polynomials}

The following theorem was proved in \cite{KLSconspline}.

\begin{theorem}[\cite{KLSconspline}*{Theorem 2.1}] \label{thm-convex-spline}
Given $r\in\N$, there is a constant $c=c(r)$ such that if $f\in C^r[-1,1]$ is convex, then there is a number $\NN=\NN(f,r)$, depending on $f$ and $r$, such that for $n\ge\NN$, there are convex piecewise polynomials $S$ of degree $r+1$ with knots at the Chebyshev partition $T_n$ (\ie $S\in \Sigma_{r+2,n} \cap\con$), satisfying
\be\label{intersplinepub}
|f(x)-S(x)|\le c(r)\left( \varphi(x)/n\right)^r \w_2\left(f^{(r)}, \varphi(x)/n\right),\quad x\in[-1,1],
\ee
and, moreover, for $x\in [-1,-1+n^{-2}] \cup  [1-n^{-2}, 1]$,
\be\label{interendtwopub}
|f(x)-S(x)|\le c(r)\varphi^{2r}(x) \w_2\left(f^{(r)},  \varphi(x)/n\right)
\ee
and
\be\label{interendonepub}
|f(x)-S(x)|\le c(r)\varphi^{2r}(x) \w_1 \left(f^{(r)},  \varphi^2(x) \right) .
\ee
\end{theorem}

As was shown in  \cite{KLSconspline},  $\NN$ in the statement of \thm{thm-convex-spline}, in general, cannot be independent of  $f$.

We will now show that the following ``smooth analog'' of this result also holds.

\begin{theorem}  \label{thm-smooth}
Given $r\in\N$, there is a constant $c=c(r)$ such that if $f\in C^r[-1,1]$ is convex, then there is a number $\NN=\NN(f,r)$, depending on $f$ and $r$, such that for $n\ge\NN$, there are continuously differentiable convex piecewise polynomials $S$ of degree $r+1$ with knots at the Chebyshev partition $T_n$ (\ie $S\in \Sigma_{r+2,n}^{(1)}\cap\con $), satisfying
\ineq{intersplinepub}, \ineq{interendtwopub} and \ineq{interendonepub}.
\end{theorem}

Let $\Spl{r}{\z_m}$ denote the space of all piecewise  polynomial functions (ppf) of degree $r-1$ (order $r$) with the knots  $\z_m := (z_i)_{i=0}^m$, $a =: z_0 <z_1 < \dots <z_{n-1} < z_m := b$.
Also, the scale of the partition $\z_m$ is denoted by
\be\label{ratio}
 \scale(\z_m) := \max_{0\le j\le m-1}\frac{|J_{j\pm1}|}{|J_j|} \, ,
\ee
where $J_j := [z_{j},z_{j+1}]$. 

In order to prove \thm{thm-smooth} we need the following lemma which is an immediate corollary of a more general result in \cite{klp}.

\begin{lemma}[see \cite{klp}*{Lemma 3.8}]    \label{lemcms}
Let   $r\in\N$, $\z_m := (z_i)_{i=0}^m$, $a =: z_0 <z_1 < \dots <z_{m-1} < z_m := b$ be a partition of $[a,b]$, and let $s\in\con \cap \Spl{r+2}{\z_m}$.
Then, there exists $\tilde s\in\con \cap \Spl{r+2}{\z_m}\cap C^1[a,b]$ such that, for any $1\leq j\leq m-1$,
\be \label{con}
\norm{s-\tilde s}{[z_{j-1}, z_{j+1}]}\le c(r,\scale(\z_m)) \omega_{r+2}(s,  z_{j+2}- z_{j-2} ;  [z_{j-2}, z_{j+2}]) \, ,
\ee
where $z_j := z_0$, $j<0$ and $z_j:= z_m$, $j>m$.
Moreover,
\be\label{last}
\tilde s^{(\nu)}(a) =s^{(\nu)}(a) \quad \mbox{\rm and}\quad \tilde s^{(\nu)}(b) =s^{(\nu)}(b)\,, \quad \nu=0,1\, .
\ee
\end{lemma}

\begin{proof}[Proof of \thm{thm-smooth}]
Let $n$ be a sufficiently large fixed number, and let  $S_0\in \Sigma_{r+2,n} \cap\con$ be a piecewise polynomial from the statement of \thm{thm-convex-spline} for which estimates \ineq{intersplinepub}--\ineq{interendonepub} hold.
Let $a := x_{2n-1, 2n}$,  $b:= x_{1, 2n}$ and let $\z_{n} := (z_i)_{i=0}^{n}$ be such that $z_0:=a$, $z_n:=b$  and $z_i := x_{n-i}$, $1\le i \le n-1$ (note that $\z_{n} \subset T_{2n}$).
Clearly, $S_0\in \Spl{r+2}{\z_{n}}$, $\scale(\z_{n}) \sim 1$, and \lem{lemcms} implies that
 there exists $\tilde S_0\in\con \cap \Spl{r+2}{\z_{n}}\cap C^1[a,b]$ such that, for any $1\leq j\leq n$,
\be \label{con1}
\norm{S_0-\tilde S_0}{\tilde I_j}\le c(r) \omega_{r+2}(S_0,  h_{j} ; \J_j) \, ,
\ee
where $\tilde I_j := I_j \cap [a,b]$ and  $\J_j := [x_{j+2}, x_{j-2}]  \cap [a,b]$, and
\be\label{last1}
\tilde S_0^{(\nu)}(a) =S_0^{(\nu)}(a) \quad \mbox{\rm and}\quad \tilde S_0^{(\nu)}(b) =S_0^{(\nu)}(b)\,, \quad \nu=0,1\, .
\ee
We now define
\[
S (x) :=
\begin{cases}
S_0(x) , & \text{if $x\in [-1,1]\setminus [a,b]$}, \\
\tilde S_0(x) ,  & \text{if $x \in  [a,b]$}.
\end{cases}
\]
Clearly,  $S\in \Sigma_{r+2,2n}^{(1)}\cap\con $, estimates \ineq{interendtwopub} and \ineq{interendonepub} hold (with $n$ replaced by $2n$), and \ineq{intersplinepub} also holds (clearly, it does not matter if we use $n$ or $2n$ there) since
 $\varphi(x)/n \sim h_j$, for any $x \in \J_j$, $1\le j \le n$. Thus, for $x\in \tilde I_j$, $1\le j\le n$,
 \begin{align*}
 |f(x)-S(x)| & \le |f(x)-S_0(x)| + \norm{S_0-\tilde S_0}{\tilde I_j} \le c \norm{f-S_0}{\J_j}    + c \omega_{r+2}(f,  h_{j} ; \J_j) \\
 & \le c h_j^r \w_2(f^{(r)}, h_j)  \le c \left( \varphi(x)/n\right)^r \w_2\left(f^{(r)}, \varphi(x)/n\right) .
 \end{align*}
\end{proof}

\begin{remark}
It follows from \cite{klp}*{Corollary 3.7} that \thm{thm-smooth} is, in fact, valid for splines of minimal defect, \ie for $S\in\Sigma_{r+2,n}\cap C^{r}\cap \con$.
\end{remark}

\sect{ Proof of \thm{thm1}}\label{sec555}

\subsection{The case for $r\ge 2$}

Let $S$ be the piecewise polynomial from the statement of \thm{thm-smooth}. Without loss of generality, we can assume that $S$ does not have knots at $x_1$ and $x_{n-1}$ (it is sufficient to treat $S$ as a piecewise polynomial with knots at the Chebyshev partition $T_{2n}$).
 Then,
\[
l_1(x) := S(x)\big|_{I_1\cup I_2}=f(1)+\frac{f'(1)}{1!}(x-1)+\dots+\frac{f^{(r)}(1)}{r!}(x-1)^r+a_+(n;f)(x-1)^{r+1}
\]
and
\[
l_n(x) := S(x)\big|_{I_n\cup I_{n-1}}=f(-1)+\frac{f'(-1)}{1!}(x+1)+\dots+\frac{f^{(r)}(1)}{r!}(x+1)^r+a_-(n;f)(x+1)^{r+1},
\]
where  $a_+(n,f)$ and $a_-(n,f)$  are some constants that depend  only on $n$ and $f$.

We will now show that
\be \label{constants}
n^{-2} \max\{ |a_+(n,f)|, |a_-(n,f)| \} \to 0 \quad \text{as} \quad n \to\infty .
\ee
Indeed, it follows from \ineq{interendonepub} that, for all $x\in I_1\cup I_2$,
\begin{align*}
|a_+(n,f)(1-x)| & \le \frac{|l_1(x) - f(x)|}{(1-x)^r}  + \frac{1}{(1-x)^r} \left| f(x) - f(1)-\frac{f'(1)}{1!}(x-1)-\dots-\frac{f^{(r)}(1)}{r!}(x-1)^r \right| \\
& \le c \w_1(f^{(r)}, 1-x) + \frac{1}{(r-1)! (1-x)^r}  \left| \int_x^1  \left( f^{(r)}(t) - f^{(r)}(1) \right) (t-x)^{r-1} dt \right| \\
& \le c \w_1(f^{(r)}, 1-x) ,
\end{align*}
and, in particular,
$n^{-2} |a_+(n,f)| \le c \w_1(f^{(r)}, n^{-2}) \to 0$  as $ n \to\infty$. Analogously, one draws a similar conclusion for $|a_-(n,f)|$.

For $f\in  C^r$, $r\ge 2$, let $i_+\geq 2$, be the smallest integer $2\leq i \leq r$, if it exists, such that $f^{(i)}(1)\neq 0$, and denote
\[
D_+(r,f) :=
\begin{cases}
(2   r!)^{-1} |f^{(i_+)}(1)|, &    \mbox{\rm if $i_+$ exists,} \\
0, & \mbox{\rm otherwise.}
\end{cases}
\]
Similarly, let $i_-\geq 2$, be the smallest integer $2\leq i \leq r$, if it exists, such that $f^{(i)}(-1)\neq 0$, and denote
\[
D_-(r,f) :=
\begin{cases}
(2 r!)^{-1} |f^{(i_-)}(1)|, &    \mbox{\rm if $i_-$ exists,} \\
0, & \mbox{\rm otherwise.}
\end{cases}
\]

Hence, if $n$ is sufficiently large, then
\be\label{der1}
S''(x)\ge D_+(r,f)   (1-x)^{r-2},\quad x\in (x_2,1],
\ee
and
\be\label{der2}
S''(x)\ge D_-(r,f)  (x+1)^{r-2},\quad x\in[-1,x_{n-2}).
\ee

\medskip


\begin{proof}[\bf Proof of  \thm{thm1} in the case $r\ge 2$] Given $r\in\N$, $r\ge 2$,  and a convex $f\in C^{r}$, let $\psi\in\Phi^2$
be such that $\omega_2(f^{(r)},t) \sim \psi(t)$,   denote $\phi(t):=t^r\psi(t)$, and note that $\phi\in\Phi^{r+2}$.

For a sufficiently large $\NN \in \N$ and  each $n\ge\NN$, we take the piecewise polynomial $S\in\Sigma_{r+2,n}$ of   \thm{thm-smooth} satisfying \ineq{der1} and \ineq{der2},
 and   observe that
\[
\omega_{r+2}(f,t)\le t^r\omega_2(f^{(r)},t) \sim \phi(t),
\]
so that by   \lem{b_k<c} with $k=r+2$, we conclude that
\[
b_{r+2}(S,\phi)\le c.
\]
Now,   it follows from \ineq{der1} and \ineq{hj}  that
\[
\min_{x\in I_2} S''(x)\geq D_+(r,f) |I_1|^{r-2} \geq 3^{-r+2} D_+(r,f) |I_2|^{r-2}
\]
and, similarly, \ineq{der2} yields
\[
\min_{x\in I_{n-1}} S''(x)  \geq 3^{-r+2} D_-(r,f) |I_{n-1}|^{r-2} .
\]
Hence, using \thm{step1111} with $k=r+2$, $d_+ := 3^{-r+2} D_+(r,f)$,  $d_- := 3^{-r+2} D_-(r,f)$ and $\alpha= 2k-2= 2r+2$, we conclude that there exists a polynomial $P\in \Poly_{cn} \cap \con$ such that
\be\label{estim}
|S(x)-P(x)|\le c\delta_n^{2r+2} (x) \rho^r_n(x)\psi(\rho_n(x)) ,\quad x\in[-1,1].
\ee

In particular, for $x\in I_1\cup I_n$, $x\neq-1,1$, using the fact that $\rho_n(x) \sim n^{-2}$ for these $x$, and $t^{-2}\psi(t)$ is nonincreasing  we have
\begin{align}\label{ends}
|S(x)-P(x)|&\le c(n\varphi(x))^{2r+2}\rho^r_n(x)\psi(\rho_n(x))\\ \nonumber
& \le c n^2 \varphi^{2r+2}(x)    \left( \frac{n\rho_n(x)}{\varphi(x)}\right)^2 \psi\left(\frac{\varphi(x)}{n} \right) \\ \nonumber
&\le c\varphi^{2r}(x)\omega_2\left(f^{(r)},\frac{\varphi(x)}n\right).
\end{align}
In turn, this implies for $x\in I_1\cup I_n$, that
\[
|S(x)-P(x)|\le c\left(\frac{\varphi(x)}n\right)^r\omega_2\left(f^{(r)},\frac{\varphi(x)}n\right),
\]
which combined with \ineq{estim} implies
\be\label{endss}
|S(x)-P(x)|\le c\left(\frac{\varphi(x)}n\right)^r\omega_2\left(f^{(r)},\frac{\varphi(x)}n\right),\quad x\in[-1,1].
\ee
Now, \ineq{endss} together with \ineq{intersplinepub} yield \ineq{interpol}, and \ineq{ends} together with \ineq{interendtwopub} yield \ineq{interpol1}.
In order to prove \ineq{interpolaux}, using the fact that $t^{-1} \w_1(f^{(r)}, t)$ is nonincreasing we have, for $x\in I_1\cup I_n$, $x\neq-1,1$,
\begin{align}\label{endsaux}
|S(x)-P(x)|&\le c(n\varphi(x))^{2r+2}\rho^r_n(x)\w_1(f^{(r)},\rho_n(x))\\ \nonumber
& \le c n^2 \varphi^{2r+2}(x)     \frac{\rho_n(x)}{\varphi^2(x)} \w_1(f^{(r)}, \varphi^2(x) )    \\ \nonumber
&\le c\varphi^{2r}(x)\omega_1\left(f^{(r)},\varphi^2(x)\right),
\end{align}
which together with \ineq{interendonepub} completes the proof of \thm{thm1}.
\end{proof}

\subsection{The case for $r=1$}

It is possible to show   that, in order to prove \thm{thm1} in the case $r=1$, it is sufficient to construct a convex polynomial $P_n$ that approximates the quadratic spline $S$ from
\thm{thm-convex-spline} (with $r=1$) so that
\[
|S(x)-P_n(x)| \le c \w_3(f, \rho_n(x)) ,
\]
and
\be \label{intder}
P_n(\pm 1)=S(\pm 1) \andd P_n'(\pm 1)=S'(\pm 1) .
\ee
In order to construct such a polynomial, one can use exactly the same method as in \cite{K} (using $S$ instead of $f$ and $2n$ instead of $n$) with the only difference that extra factors $(1-y^2)$ should appear inside integrals in the definitions of $Q_j$, $\overline Q_j$ and $T_j$ (see \cite{K}*{p. 158}).
All estimates are then still valid and either follow from the statements in Section~\ref{auxresults} or are proved similarly. This change is needed in order to guarantee that \ineq{intder} holds because addition of these factors implies that the first derivatives of these modified polynomials $Q_j$, $\overline Q_j$ and $T_j$ are $0$ at $\pm 1$ which, in turn, implies that the first derivatives of other auxiliary polynomials $\sigma_j$, $R_j$ and $\overline R_j$ at $\pm 1$ are ``correct''.
 We omit details.


\section{Appendix: proof of \thm{step1111}}

Throughout the proof, we fix $\beta :=k+7$ and   $\gamma:=\beta_1 -1 = 60(\alpha+\beta)+k$. Hence, the constants $C_1,\dots,C_6$ (defined below)  as well as the constants $C$, may depend only on $k$ and $\alpha$. We also note that $S$ does not have to be twice differentiable at the Chebyshev knots $x_j$. Hence, when we write $S''(x)$ (or $S_i''(x)$, $1\leq i \leq 4$) everywhere in this proof, we implicitly assume that $x\neq x_j$, $1\leq j\leq n-1$.

Let $C_1 :=C$, where the constant   $C$ is taken from \ineq{Qu} (without loss of generality we assume that $C_1\leq 1$), and let $C_2:=C$  with $C$ taken from \ineq{7.24'}.
 We also fix an integer $C_3$ such that
\be\label{8.4}
C_3\ge 8k/C_1 .
\ee
Without loss of generality, we may assume that  $n$ is divisible by $C_3$, and put $n_0:=  n/C_3$.

We divide $[-1,1]$ into $n_0$ intervals
\[
E_q:=[x_{qC_3},x_{(q-1)C_3}]=I_{qC_3}\cup\dots\cup I_{(q-1)C_3+1}, \quad 1\leq q\leq n_0 ,
\]
consisting of $C_3$ intervals $I_i$ each (\ie $m_{E_q} = C_3$, for all $1\leq q\leq n_0$).

We   write ``$j\in UC$'' (where ``$UC$'' stands for ``\underline{U}nder \underline{C}ontrol")  if there is
$x_j^*\in(x_j,x_{j-1})$  such that
\be\label{8.51}
S''(x_j^*)\le \frac{5C_2\phi(\rho_n (x_j^*))}{\rho_n^2(x_j^*)} .
\ee

We   say that $q\in G$ (for ``\underline{G}ood), if the interval $E_q$ contains at least $2k-5$
intervals $I_j$ with $j\in UC$.
%

Then,   \ineq{8.51} and \lem{newlemmasec101}
imply that,
\be\label{8.71}
S''(x)\le \frac{C\phi(\rho)}{\rho^2},\quad x\in E_q , \; q\in G.
\ee
Set
\[
E:=\cup_{q\notin G}E_q,
\]
and decompose $S$ into a ``small" part and a ``big" one, by setting
$$
s_1(x):=\begin{cases}  S''(x),&\quad\text{if}\quad x\notin E,\\
0,&\quad\text{otherwise},\end{cases}.
$$
and
\[
s_2 (x):=S''(x)-s_1(x) =
\begin{cases}
0,&\quad\text{if}\quad x\notin E,\\
 S''(x),&\quad\text{otherwise},
 \end{cases}
\]
and putting
\[
S_1(x) :=S(-1)+(x+1)S'(-1)+\int_{-1}^x(x-u)s_1(u)du \andd
S_2(x) :=\int_{-1}^x(x-u)s_2(u)du.
\]
(Note that $s_1$ and $s_2$ are well defined for $x\neq x_j$,
$1\le j\le n-1$, so that $S_1$ and $S_2$ are well defined
everywhere and possess second derivatives   for $x\neq x_j$,
$1\le j\le n-1$.)


Evidently,
$$
S_1,S_2\in\Sigma_{k,n}^{(1)},
$$
and
\[
S''_1(x)\ge0\,\text{ and }\,S''_2(x)\ge0, \quad x\in[-1,1].
\]
Now, \ineq{8.71} implies that
\[
S''_1(x)\le \frac{C\phi(\rho)}{\rho^2},\quad x\in[-1,1],
\]
which, in turn, yields by \lem{important},
\[
b_k(S_1, \phi)\le C.
\]
Together with \ineq{d11}, we obtain
\be\label{8.81}
b_k(S_2, \phi)\le b_k(S_1, \phi) + b_k(S, \phi) \le
C+1\le\lceil C+1\rceil=:C_4.
\ee

The set $E$ is a union of disjoint intervals $F_p=[a_p,b_p]$,
between any two of which, all intervals $E_q$ are with $q\in G$.
We may assume that $n>C_3C_4$, and write $p\in AG$ (for ``\underline{A}lmost
\underline{G}ood"), if $F_p$ consists of no more than $C_4$ intervals $E_q$,
 that is, it consists of no more than $C_3C_4$ intervals
$I_j$. Hence, by \lem{newlemmasec101},
\be\label{4.71}
S_2''(x)\le\frac {C\,\phi(\rho)}{\rho^2},\quad x\in F_p,\;  p\in AG.
\ee

One may   think of intervals $F_p$, $p\not\in AG$, as  ``long'' intervals where $S''$  is   ``large'' on many subintervals $I_i$  and  rarely dips down to $0$.
Intervals $F_p$, $p \in AG$, as well as all intervals $E_q$ which are not contained in any $F_p$'s (\ie all ``good'' intervals $E_q$)  are where $S''$ is ``small' in the sense that the inequality $ S''(x)\le  {C\phi(\rho)}/{\rho^2}$ is valid there.

Set
\[
F:=\cup_{p\notin AG}F_p,
\]
note that $E =  \cup_{p\in AG}F_p \cup F$,
and decompose $S$ again by setting
$$
s_4:=\begin{cases} S''(x),&\quad\text{if}\quad x\in F,\\
0,&\quad\text{otherwise},
\end{cases}
$$
and
\[
s_3(x) := S''(x)-s_4(x) =
\begin{cases} 0,&\quad\text{if}\quad x\in F,\\
S''(x),&\quad\text{otherwise},
\end{cases}
\]
and putting
\be\label{s341}
S_3(x) :=S(-1)+(x+1)S'(-1)+\int_{-1}^x(x-u)s_3(u)du \andd
S_4(x) :=\int_{-1}^x(x-u)s_4(u)du.
\ee

Then, evidently,
\be\label{8.91}
S_3,S_4\in\Sigma_{k,n}^{(1)}, \quad S_3+S_4 = S,
\ee
and
\be\label{8.101}
S''_3(x)\ge0\,\text{ and }\,S''_4(x)\ge0, \quad x\in[-1,1].
\ee

We remark that, if $x\not\in  \cup_{p\in AG} F_p$, then $s_1(x)=s_3(x)$ and $s_2(x)=s_4(x)$. If
$x \in  \cup_{p\in AG} F_p$, then $s_1(x)=s_4(x)=0$ and $s_2(x)=s_3(x)=S''(x)$.

For $x\in\cup_{p\in AG} F_p$, \ineq{4.71} implies that
$$
S_3''(x)=S_2''(x)\le\frac {C\,\phi(\rho)}{\rho^2}.
$$
For all other $x$'s,
$$
S_3''(x)=S_1''(x)\le\frac {C\,\phi(\rho)}{\rho^2}.
$$
We conclude that
\be\label{8.121}
S_3''(x)\le \frac {C_5\,\phi(\rho)}{\rho^2},\quad x\in[-1,1],
\ee
which by virtue of  \lem{important}, yields that $b_k(S_3,\phi)\le C$. As above, we obtain
\be\label{8.131}
b_k(S_4,\phi)  \leq b_k(S_3,\phi) + b_k(S,\phi) \le C+1\le\lceil C+1\rceil=:C_6.
\ee

We will approximate $S_3$ and $S_4$ by convex  polynomials that achieve the required degree of pointwise approximation.

\medskip

\noindent {\bf Approximation of $S_3$:}


If $d_+>0$, then there exists $\NN^*\in\N$, $\NN^* = \NN^* (d_+, \psi)$, such that, for $n> \NN^*$,
\[
\frac {\phi(\rho)}{\rho^2}=\rho^{r-2}\psi(\rho)<\frac{d_+|I_2|^{r-2}}{C_5}\le C^{-1}_5\,S''(x),\quad x\in I_2,
\]
where the first inequality follows since $\psi(\rho) \leq \psi(2/n) \to0$  as $n\to\infty$, and the second inequality follows by \ineq{d+11}.
Hence, by \ineq{8.121}, if $n>\NN^*$, then $s_3(x)\ne S''(x)$ for $x\in I_2$.  Therefore, since $s_3(x)=S''(x)$, for all $x\notin F$, we conclude that $I_2\subset F$, and so   $E_1\subset F$,
and $s_3(x)=0$, $x\in E_1$. In particular, $s_3(x)\equiv0$, $x\in I_1$.

Similarly, if $d_->0$, then  using \ineq{d-1} we conclude that there exists $\NN^{**}\in\N$, $\NN^{**} = \NN^{**}(d_-,\psi)$, such that, if  $n>\NN^{**}$, then $s_3(x)\equiv0$ for all $x\in I_n$.

Thus, we conclude that for $n\ge\max\{\NN^*,\NN^{**}\}$, we have
\be \label{endpoints1}
s_3(x)=0,\quad \mbox{\rm for all }\; x\in I_1\cup I_n.
\ee

Therefore, in view of \ineq{8.91} and \ineq{8.101}, it follows by   \lem{step11} combined with \ineq{8.121} that, in the case $d_+>0$ and $d_->0$, there exists
a convex polynomial $r_n \in \Poly_{Cn}$   such that
\be\label{8.141}
|S_3(x)-r_n(x)|\le C\,\delta^{\alpha} \phi(\rho),\quad x\in[-1,1].
\ee

Suppose now that $d_+=0$ and $d_->0$. First, proceeding as above, we conclude that $s_3\equiv 0$ on $I_n$.
Additionally, if $E_1 \subset F$, then, as above,  $s_3\equiv 0$ on  $I_1$ as well. Hence, \ineq{endpoints1} holds which, in turn, implies \ineq{8.141}.


If $E_1 \not\subset F$, then $s_3(x)=S''(x)$, $x\in I_1$, and so it follows from (7.3) that, for some constant $A_1\ge0$,
\[
s_3(x)=S''(x)=A_1(1-x)^{k-3},\quad x\in I_1.
\]
Note that $A_1$ may depend on $n$, but by (7.16) we conclude that,
\[
A_1\le C_5\frac{\phi(\rho_n(x_1))}{(1-x_1)^{k-3}\rho^2(x_1)}\sim n^{2k-2}\phi(n^{-2}).
\]
Hence, for $x\in I_1$,
\be\label{7.20}
S''_3(x)=s_3(x)=A_2n^{2k-2}\phi(n^{-2})(1-x)^{k-3},
\ee
where   $A_2$ is a nonnegative constant that may depend on $n$ but $A_2\le C$.

We  now construct $\wS_3 \in \Sigma_{k, 2n}$ which satisfies all conditions of \lem{step11} (with $2n$ instead of $n$).
Note  that $x_j := x_{j,n} = x_{2j,2n}$, denote $\xi := x_{1,2n}$ and define
\[
\widetilde S_3(x) :=
\begin{cases}
S_3(x), & \text{if $x<x_1$}, \\
S_3(1)+(x-1)S_3'(1)=:L(x), & \text{if $\xi  < x \le 1$}, \\
\ell(x), & \text{if $x\in [x_1,\xi]$},
\end{cases}
\]
where $\ell(x)$ is the linear polynomial   chosen so  that $\wS_3$ is continuous on $[-1,1]$, \ie   $\ell(x_1)=S_3(x_1)$ and $\ell(\xi)= S_3(1)+(\xi-1)S_3'(1) = L(\xi)$.
Clearly, $\wS_3 \in C[-1,1]$ (and, in fact, is in $C^1[-1,x_1)$), $\wS_3''(x) \leq C\rho^{-2}\phi(\rho)$, $x\not\in \{x_j\}_{j=1}^{n-1}\cup \{\xi\}$,  and  $\widetilde S_3''\equiv 0$ on $I_{1,2n}\cup I_n$. Note that $\wS_3'$ may be discontinuous at $x_1$ and $\xi$,  but, evidently, the slope of $L$ is no less than the slope of $\ell$, so that $\wS_3$ is convex in $[x_1,1]$.

Denote
\[
S_L:=S_3-L, \quad \wS_L:=\wS_3-L \andd \ell_L:=\ell-L,
\]
and note that,
\[
\wS_L(x)\equiv 0, \quad x\in[\xi,1], \andd \wS_L(x)=\ell_L(x), \quad x\in[x_1,\xi].
\]
Also, $\ell_L(x_1)=S_L(x_1)$ and $\ell_L(\xi)=\wS_L(\xi)=0$, and in view of \ineq{7.20},
\begin{align}\label{7.21}
0\le S_L(x)&=S_3(x)-S_3(1)-(x-1)S'_3(1)\\
&=\int_x^1(u-x)s_3(u)du=A_3n^{2k-2}\phi(n^{-2})(1-x)^{k-1},\quad x\in I_1,\nonumber
\end{align}
where $A_3\le C$. Now, the tangent line to $S_L$ at $x=x_1$ is
\[
y(x)=S_L(x_1)+(x-x_1)S'_L(x_1)=A_3n^{2k-2}\phi(n^{-2}) (1-x_1)^{k-2}   \left(1-x_1 -(k-1)(x-x_1)  \right),
\]
which intersects the $x$ axis at
$$
x_1+\frac{1-x_1}{k-1}\le\frac{1+x_1}2<x_{1,2n}=\xi.
$$
Hence, the slope of $\ell_L$ is no less than the slope of that tangent and, in turn, we conclude that the slope of $\ell$ is no less than $S'_3(x_1)$, so that $\wS_3$ is convex in $[-1,1]$.

Further, we have,
\begin{align}\label{7.23}
|S_3(x)-\wS_3(x)|&=|S_L(x)-\wS_L(x)|\le S_L(x)+\wS_L(x)\\
&\le S_L(x_1)+\wS_L(x_1)=2S_L(x_1)\le C \phi(n^{-2}),\quad x\in[x_1,\xi],\nonumber
\end{align}
and
\begin{align}\label{7.24}
|\wS_3(x)-S_3(x)|=|L(x)-S_3(x)|&=S_L(x)=A_3\phi(n^{-2})n^{2k-2}(1-x)^{k-1}\\
&\le C\delta^{2k-2}\phi(n^{-2}),\quad x\in[\xi,1].\nonumber
\end{align}
Note that  $\wS'_3$ may have (nonnegative) jumps at $x_1$ and $\xi$. However,
\be\label{7.25}
\wS_3'(\xi+) - \wS_3'(\xi-)+\wS_3'(x_1+) - \wS_3'(x_1-)=S'_3(1)-S_3'(x_1)=-S'_L(x_1)\le Cn^2\phi(n^{-2}),
\ee
so that \lem{step11} implies that there exists a convex polynomial
$r_n \in \Poly_{Cn}$ such that,
\[
|\widetilde S_3(x)-r_n(x)|\le C\,\delta^{\alpha} \phi(\rho),\quad x\in[-1,1].
\]
Observing that $\wS_3\equiv S_3$ on $[-1,x_1]$, and combining with \ineq{7.23} and \ineq{7.24} (recalling that $n^{-2}\le\rho$), we conclude that
\[
\left| \wS_3(x) - S_3(x) \right| \leq C\delta^{2k-2} \phi(\rho), \quad x\in [-1,1] ,
\]
so that
\be \label{tildeaux1}
| S_3(x)-r_n(x)|\le C\,\delta^{\min\{\alpha, 2k-2\}} \phi(\rho), \quad x\in[-1,1].
\ee

Finally, if $d_-=0$ and $d_+>0$, then the considerations are completely analogous and, if $d_-=0$ and $d_+=0$, then $\wS_3$ can be modified further on $I_n$ using \ineq{d-is01} and the above argument.

Hence, we've constructed a convex  polynomial $r_n \in \Poly_{Cn}$   such that, in the case when both $d_+$ and $d_-$ are strictly positive, \ineq{8.141} holds, and \ineq{tildeaux1} is valid if at least one of these numbers is $0$.

\medskip

\noindent
{\bf Approximation of $S_4$:}

Given a set $A\subset  [-1,1]$, denote
\[
A^e:=\cup_{I_j\cap A\ne\emptyset}I_j \andd  A^{2e}:=(A^e)^e,
\]
where $I_0=\emptyset$ and $I_{n+1}=\emptyset$.  For example, $[x_7,x_3]^e=[x_8,x_2]$, $I_1^e=I_1\cup I_2$, etc.

Also, given subinterval  $I\subset [-1,1]$ with its endpoints at the Chebyshev knots, we   refer to the right-most and the left-most intervals $I_i$ contained in $I$ as $EP_+(I)$ and $EP_-(I)$, respectively (for the ``\underline{E}nd \underline{P}oint'' intervals). More precisely, if $1\leq \mu<\nu \leq n$ and
\[
I = \bigcup_{i=\mu}^\nu I_i ,
\]
 then $EP_+(I) := I_\mu$, $EP_-(I) := I_\nu$ and $EP(I):= EP_+(I)  \cup EP_-(I) = I_\mu \cup I_\nu$. For example, $EP_+[-1,1] := I_1$, $EP_-[-1,1] := I_n$, $EP_+[x_7, x_3] = [x_4, x_3] = I_4$, $EP_-[x_7, x_3] = [x_7, x_6] = I_7$,
$EP [x_7, x_3] = I_4\cup I_7$,  etc.
Here, we simplified the notation by using $EP_\pm [a,b] := EP_\pm ([a,b])$ and $EP  [a,b] := EP ([a,b])$.

In order to approximate $S_4$, we observe that for $p\notin AG$,
\[
S_4''(x)=S_2''(x),\quad x\in F^{2e}_p,
\]
so that by virtue of \ineq{8.81}, we conclude that
\be\label{8.16}
b_k(S_4,\phi, F^{2e}_p)=b_k(S_2, \phi, F^{2e}_p)\le b_k(S_2, \phi)\le C_4.
\ee
(Note that, for $p\in AG$, $S_4$ is  linear  in $F_p^{2e}$ and so $b_k(S_4,\phi, F_p^{2e})=0$.)

We will approximate $S_4$ using the polynomial $D_{n_1}(\cdot,S_4) \in \Poly_{Cn_1}$  defined  in \lem{uncon}  (with $n_1:=C_6n$), and then we construct
  two ``correcting'' polynomials $\overline Q_n, M_n \in\Poly_{Cn}$ (using  \lem{QM}) in order to make sure that the resulting approximating polynomial is convex.

We begin with $\overline Q_n$. For each $q$ for which $E_q\subset F$, let
$J_q$ be the union of all intervals $I_j\subset E_q$ with $j\in
UC$ with the union of both intervals $I_j\subset E_q$  at the endpoints of $E_q$.
In other words,
\[
J_q := \bigcup_j \left\{ I_j \st j\in UC \andd  I_j\subset E_q\right\}  \;  \cup \;  EP(E_q) .
\]

Since $E_q\subset F$, then $q\notin G$ and so the number   of
 intervals $I_j\subset E_q$ with $j\in UC$  is at most $2k-6$.
Hence, by \ineq{8.4},
\[
m_{J_q}\le 2k-4<  2k   \le \frac{C_1C_3}4\le\frac{C_3}4,
\]
Recalling that the total number $m_{E_q}$ of intervals $I_j$ in $E_q$ is $C_3$ we conclude that \lem{QM} can be used with $E:= E_q$ and $J:=J_q$.
Thus, set
\[
\overline Q_n:=\sum_{q\, :\; E_q\subset F} Q_n(\cdot,E_q,J_q),
\]
where $Q_n$ are   polynomials from \lem{QM}, and denote
\[
J:=\bigcup_{q\, :\; E_q\subset F}J_q.
\]

Then, \ineq{Qu} through \ineq{Qu2} imply that that $\overline Q_n$ satisfies
\begin{align}\label{8.17}
\text{(a)} \qquad \overline Q_n''(x)&\ge 0,\quad x\in[-1,1]\setminus F,\nonumber\\
\text{(b)} \qquad \overline Q_n''(x)&\ge -   \frac{\phi(\rho)}{\rho^2}\quad x\in F\setminus
J,\\
\text{(c)} \qquad \overline Q_n''(x)&\ge  4  \frac{ \phi(\rho)}{\rho^2}\delta^{8\alpha},\quad x\in J.\nonumber
\end{align}
Note that  the inequalities in \ineq{8.17} are valid since, for any given $x$,
all relevant $Q_n''(x,E_q,J_q)$, except perhaps one, are nonnegative, and
\[
C_1 \frac{m_{E_q}}{m_{J_q}} \geq \frac{C_1C_3}{2k} \ge 4.
\]

Also, it follows from \eqref{Qu2} that, for any $x\in[-1,1]$,
\begin{eqnarray} \label{8.20}
|\overline Q_n(x)| & \leq & C\delta^\alpha \rho \phi(\rho)    \sum_{q\, :\; E_q\subset F} \sum_{j: \, I_j\subset  E_q}\frac{h_j}{(|x-x_j|+\rho)^2} \\ \nonumber
& \leq &
C\delta^\alpha \rho \phi(\rho)  \sum_{j=1}^n \frac{h_j}{(|x-x_j|+\rho)^2} \\ \nonumber
& \leq &
C\delta^\alpha \rho \phi(\rho) \int_0^\infty  \frac{du}{(u+\rho)^2}    \\ \nonumber
& = &
C\delta^\alpha \phi(\rho).
\end{eqnarray}

Next, we define the polynomial $M_n$. For each $F_p$ with $p\notin AG$, let $J_p^-$ denote the union of the two intervals on the left
side of $F^e_p$ (or just the interval $I_n$ if $-1\in F_p$), and let $J_p^+$ denote
the union of the two intervals on the right side of
$F^e_p$ (or just one interval $I_1$ if  $1\in F_p$), \ie
\[
J_p^- = EP_-(F^e_p) \cup EP_-(F_p) \andd J_p^+ = EP_+(F^e_p) \cup EP_+(F_p).
\]
Also, let $F_p^-$ and $F_p^+$ be the closed intervals each consisting of $m_{F_p^\pm}:=C_3C_4$
intervals $I_j$ and such that $J_p^-\subset F_p^-\subset F^e_p$ and $J_p^+\subset F_p^+\subset F^e_p$, and put
\[
J_p^*:=J_p^-\cup J_p^+ \andd J^*:=\cup_{p\notin AG}J_p^*.
\]
Now, we set
\[
M_n:=\sum_{p\notin AG}\left(Q_n(\cdot,F_p^+,J_p^+) +
Q_n(\cdot,F_p^-,J_p^-)\right).
\]

Since $m_{F_p^+}=m_{F_p^-}=C_3C_4$ and $m_{J_p^+}, m_{J_p^-} \leq 2$, it follows from \ineq{8.4} that
\[
C_1 \min\left\{ \frac{m_{F_p^+}}{m_{J_p^+}} ,\frac{m_{F_p^-}}{m_{J_p^-}} \right\} \ge  \frac{C_1 C_3 C_4}{2}
\ge 2C_4.
\]
Then    \lem{QM} implies
\be\label{8.24}
|M_n(x)|\le C\,\delta^{\alpha}\phi(\rho)
\ee
(this follows from \ineq{Qu2} using the same sequence of inequalities that was used to prove \ineq{8.20} above), and
\begin{align}\label{8.21}
\text{(a)} \quad M_n''(x)&\ge-2\frac{\phi(\rho)}{\rho^2},\quad x\in F\setminus J^*,\nonumber\\
\text{(b)} \quad M_n''(x)&\ge
2C_4         \, \delta^{8\alpha}   \frac{\phi(\rho)}{\rho^2},\quad x\in J^*,\\
\text{(c)} \quad M_n''(x)&\ge
2C_4\, \delta^{8\alpha}\frac{\phi(\rho)}{\rho^2}\left(\frac\rho{\dist\,(x,F)}
\right)^{\gamma+1},\; x\in[-1,1]\setminus F^e,\nonumber
\end{align}
where in the last inequality we used the fact that 
\[
\max\{\rho,\dist\,(x,F^e)\}\le  \dist\,(x,F),\quad
x\in[-1,1]\setminus F^e,
\]
which follows from \ineq{newauxest}.

The third auxiliary polynomial is
  $D_{n_1}:=D_{n_1}(\cdot, S_4)$ with $n_1 = C_6n$ from \lem{uncon}. By \ineq{8.131}, \ineq{7.23'} yields
\be\label{8.25}
|S_4(x)-D_{n_1}(x)|\le C\,\delta^\gamma \phi(\rho) \le C\,\delta^\alpha \phi(\rho)     ,\quad x\in[-1,1],
\ee
since $\gamma > \alpha$, and \ineq{7.24'}  implies that, for any
interval $A\subset [-1,1]$ having  Chebyshev knots as endpoints,
\begin{align}\label{8.26}
|S_4''(x)-D_{n_1}''(x)|&\le C_2\, \delta^{\gamma} \frac{\phi(\rho)}{\rho^2}b_k(S_4,\phi, A)\\&\quad+
C_2C_6 \, \delta^{\gamma}\frac{\phi(\rho)}{\rho^2}\frac n{n_1}
\left(\frac{\rho}{\dist(x,[-1,1]\setminus A)}\right)^{\gamma+1},\quad
x\in A.\nonumber
\end{align}

We now define
\be\label{8.27}
R_n:=D_{n_1}+C_2\overline Q_n+C_2M_n.
\ee
By virtue of \ineq{8.20}, \ineq{8.24}, and \ineq{8.25}  we obtain
$$
|S_4(x)-R_n(x)|\le C\,\delta^{\alpha}\phi(\rho),\quad x\in[-1,1],
$$
which combined with \ineq{8.141} and \ineq{tildeaux1}, proves \ineq{approx101} and \ineq{newapprox102} for $P :=R_n+r_n$.

Thus, in order to conclude the proof of  \thm{step1111}, we should prove that
$P$ is convex. We recall that $r_n$ is convex, so   it is sufficient to show that $R_n$ is convex as well.

Note that \ineq{8.27} implies
\[
R_n''(x) \ge C_2\overline Q_n''(x)+C_2M_n''(x)
 -|S_4''(x)- D_{n_1}''(x)|+S_4''(x), \quad x\in[-1,1],
\]
(this inequality is extensively used in the three cases below),
and that \ineq{8.26} holds for {\em any} interval $A$ with Chebyshev knots as the endpoints, and so we can use
different intervals $A$ for different points $x\in [-1,1]$.
We consider three cases depending on whether (i) $x\in F\setminus J^*$, or (ii)  $x\in J^*$, or (iii)~$x\in[-1,1]\setminus F^e$.

{\bf Case (i):} If $x\in F\setminus J^*$, then, for some $p\notin AG$,
  $x\in F_p\setminus J_p^*$, and so   we take $A:=F_p$. Then, the quotient
inside the  parentheses in \ineq{8.26} is bounded above by 1 (this follows from \ineq{newauxest}). Also, since   $s_4(x)=S''(x)$, $x\in F$, it follows that
$b_k(S_4,\phi, F_p)=b_k(S,\phi, F_p)\le1$. Hence,
\begin{align}\label{8.28}
|S_4''(x)-D_{n_1}''(x)|&\le  C_2\, \frac{\phi(\rho)}{\rho^2}b_k(S_4,\phi,F_p)+
C_2C_6 \, \frac{\phi(\rho)}{\rho^2}\frac n{n_1}\\
& \le
2 C_2 \, \frac{\phi(\rho)}{\rho^2},\quad x\in F\setminus J^*. \nonumber
\end{align}
Note  that $x\notin I_1\cup I_n$ (since   $F\setminus J^*$ does not contain any intervals in $EP(F_p)$, $p\notin AG$), and so $\delta = 1$.

It now follows by  \ineq{8.17}(c), \ineq{8.21}(a),  \ineq{8.28} and \ineq{8.101}, that
\[
R_n''(x)\ge C_2 \, \frac{\phi(\rho)}{\rho^2}(4-2-2)=0,\quad
x\in J\setminus J^*.
\]
If $x\in F\setminus(J\cup J^*)$, then \ineq{8.51} is violated and so
\[
S_4''(x) = S''(x) >\frac{5C_2\phi(\rho)}{\rho^2}.
\]
Hence, by virtue of \ineq{8.17}(b), \ineq{8.21}(a) and \ineq{8.28}, we get
\[
R_n''(x)\ge C_2\, \frac{\phi(\rho)}{\rho^2}(-1-2-2+5)=0,\quad
x\in F\setminus(J\cup J^*).
\]

{\bf Case (ii):}
If $x\in J^*$, then, $x\in J_p^*$, for some $p\notin AG$, and we take  $A:=F_p^{2e}$. Then,
 \ineq{8.16} and \ineq{8.26} imply (again, \ineq{newauxest} is used to estimate the quotient inside the  parentheses in \ineq{8.26}),
\begin{align}\label{8.29}
|S_4''(x)-D_{n_1}''(x)|&\le
C_2\,\delta^{\gamma} \frac{\phi(\rho)}{\rho^2}b_k(S_4,\phi, F_p^{2e})+
C_2C_6 \,\delta^{\gamma} \frac{\phi(\rho)}{\rho^2}\frac n{n_1}\\
&\le
2 C_2 C_4  \,\delta^{\gamma} \frac{\phi(\rho)}{\rho^2},\quad x\in J^*.    \nonumber
\end{align}
Now, we note that $EP(F_p) \subset J$, for all $p\notin AG$, and so $F\cap J^* \subset J$. Hence,
 using   \ineq{8.17}(a,c), \ineq{8.21}(b),
 \ineq{8.29} and \ineq{8.101}, we obtain
\[
R_n''(x)\ge 2 C_2 C_4\, \delta^{8\alpha} \frac{\phi(\rho)}{\rho^2}   - 2 C_2C_4\,\delta^{\gamma} \frac{\phi(\rho)}{\rho^2} \geq 0 ,
\]
 since $\gamma > 8\alpha$, and so $\delta^{\gamma} \leq \delta^{8\alpha}$.

{\bf Case (iii):}
If $x\in[-1,1]\setminus F^e$, then we take $A$ to be the
connected component of $[-1,1]\setminus F$ that contains $x$. Then by \ineq{8.26},
\begin{align}\label{8.30}
&|S_4''(x)-D_{n_1}''(x)|\nonumber\\
&\quad\le  C_2\, \delta^{\gamma} \frac{\phi(\rho)}{\rho^2}b_k(S_4,\phi,A)+
C_2C_6 \, \delta^{\gamma} \frac{\phi(\rho)}{\rho^2}\frac n{n_1}
\left(\frac\rho{\dist(x,[-1,1]\setminus A)}\right)^{\gamma+1}\\
&\quad = C_2 \, \delta^{\gamma} \frac{\phi(\rho)}{\rho^2}\left(\frac\rho{\dist(x,
F)}\right)^{\gamma+1},\quad x\in[-1,1]\setminus F^e , \nonumber
\end{align}
where we used the fact that $S_4$ is linear in $A$, and so $b_k(S_4,\phi,A)=0$.

Now, \ineq{8.17}(a), \ineq{8.21}(c), \ineq{8.30} and \ineq{8.101} imply,

\[
R_n''(x)\ge
\frac{\phi(\rho)}{\rho^2} \left(\frac\rho{\dist(x, F)}\right)^{\gamma+1} \left(
2 C_2 C_4 \delta^{8\alpha}   - C_2 \, \delta^{\gamma}
  \right)\geq 0,
\]
 since $C_4\geq 1$ and $\gamma > 8\alpha$.

Thus, $R_n''(x)\geq 0$ for all $x\in[-1,1]$, and so we have constructed a convex
polynomial $P$, satisfying \ineq{approx101} and \ineq{newapprox102}, for each $n\ge \NN$.
This completes the proof of \thm{step1111}.


\end{document}